\documentclass[11pt]{amsart}
\usepackage{amsfonts}
\usepackage{graphicx}
\usepackage{amsmath}
\usepackage{amsthm}
\usepackage{amssymb}
\usepackage{amscd}
\usepackage{color}

\usepackage{graphicx}
\usepackage{microtype}
\usepackage{enumitem}
\usepackage{pinlabel}
\usepackage[top=1.2in,bottom=1.4in,left=1.3in,right=1.3in]{geometry}

\theoremstyle{definition}

\newtheorem{theorem}{Theorem}[section]
\newtheorem{lemma}[theorem]{Lemma}
\newtheorem{corollary}[theorem]{Corollary}
\newtheorem{proposition}[theorem]{Proposition}
\newtheorem*{proposition*}{Proposition}

\newtheorem{problem}[theorem]{Problem}
\newtheorem{observation}[theorem]{Observation}

\theoremstyle{definition}
\newtheorem{definition}[theorem]{Definition}
\newtheorem{remark}[theorem]{Remark}
\newtheorem{notation}[theorem]{Notation}
\newtheorem{example}[theorem]{Example}
\newtheorem*{claim}{Claim}

\newtheorem*{convention}{Convention}
\newtheorem*{acknowledge}{Acknowledgements}

\newcommand{\R}{\mathbb{R}}
\newcommand{\Z}{\mathbb{Z}}
\newcommand{\N}{\mathbb{N}}

\newcommand{\id}{id}
\newcommand{\Emc}{E_{\rm mc}}

\newcommand{\sE}{\mathsf{E}}

\newcommand{\calA}{\mathcal{A}}
\newcommand{\calB}{\mathcal{B}}
\newcommand{\calC}{\mathcal{C}}
\newcommand{\calD}{\mathcal{D}}
\newcommand{\calF}{\mathcal{F}}
\newcommand{\calM}{\mathcal{M}}

\newcommand{\calV}{\mathcal{V}}

\newcommand{\calP}{\mathcal{P}}

\newcommand{\less}{\prec}
\newcommand{\lesseq}{\preccurlyeq}

\newcommand{\moreq}{\succcurlyeq}

\DeclareMathOperator{\Homeo}{Homeo}
\DeclareMathOperator{\mcg}{Map}
\DeclareMathOperator{\pmcg}{PMap}

\DeclareMathOperator{\I}{i}
\DeclareMathOperator{\diam}{diam}

\newcounter{notes}%[page]   %Le 2eme argument fait reinitialiser les numeros de notes a chaque page
\newcommand{\from}{\colon\,}
\newcommand{\CS}{{\mathcal{C}(S)}}

\renewcommand{\paragraph}[1]{\medskip 
\noindent \textbf{#1}}

\title{Large scale geometry of big mapping class groups}
\author{Kathryn Mann and Kasra Rafi}
\begin{document}

\begin{abstract}
We study the large-scale geometry of mapping class groups of surfaces of infinite type, using the framework of Rosendal for coarse geometry of non locally compact groups.   We give a complete classification of those surfaces whose mapping class groups have local 
{\em coarse boundedness} (the analog of local compactness). When the end space of the 
surface is countable or tame, we also give a classification of those surfaces where 
there exists a coarsely bounded generating set (the analog of finite or compact generation, giving the group a well-defined quasi-isometry type) and those surfaces with mapping class groups of bounded diameter
(the analog of compactness).

We also show several relationships between the topology of a surface and the geometry of its mapping class groups.  For instance, we show that {\em nondisplaceable subsurfaces} are responsible for nontrivial geometry and can be used to produce unbounded length functions on mapping class groups using a version of subsurface projection; while {\em self-similarity} of the space of ends of a surface is responsible for boundedness of the mapping class group.  
\end{abstract}

\maketitle

%- - - - - - - - - - - - - - - - - - - - - - - - - - - - - - - - - -
\section{Introduction} 

Mapping class groups of surfaces of infinite type (with infinite genus or infinitely many ends) form a rich class of examples of non locally compact Polish topological groups.   These ``big" mapping class groups can be seen as as natural generalizations of, or limit objects of, the mapping class groups of finite type surfaces, and also arise naturally in the study of laminations and foliations, and the dynamics of group actions on finite type surfaces.  

Several recent papers (see for instance \cite{DFV, BDR, AFP}) have studied big mapping class groups through their actions on associated combinatorial structures such as curve or arc complexes.   From this perspective, an important problem is to understand whether a given mapping class group admits a {\em metrically nontrivial} action on such a space, namely, an action with unbounded orbits. 
It is our observation that this should be framed as part of a larger question, one of the {\em coarse} or {\em large-scale geometry} of big mapping class groups.   This is the goal of the present work.   

However, describing the large-scale structure of big mapping class groups - or even determining whether this notion makes sense - is a nontrivial problem, as standard tools of geometric group theory apply only to locally compact, compactly generated groups, and big mapping class groups do not fall in this category.  Instead, we use recent work of Rosendal \cite{RosendalLargeScale} that extends the framework geometric group theory to a broader class of topological groups, using the notion of {\em coarse boundedness}.

\begin{definition}
Let $G$ be a Polish topological group.   
A subset $A \subset G$ is {\em coarsely bounded}, abbreviated {\em CB}, if every compatible left-invariant metric on $G$ gives $A$ finite diameter\footnote{In \cite{RosendalBergman} and related work, this condition is called (OB), for {\em orbites born\'ees}, as it is equivalent to the condition that for any continuous action of $G$ on a metric space $X$ by isometries, the diameter of every orbit $A\cdot x$ is bounded.   {\em Coarsely bounded} appears in \cite{Rosendal_book}, we prefer this terminology as it is more suggestive of the large-scale geometric context.  
}.  
\end{definition} 
\noindent For example, in a locally compact group, the CB sets are precisely the compact ones.
Rosendal shows the following. 

\begin{theorem}[Rosendal  \cite{RosendalLargeScale}]
Let $G$ be a Polish group that has both a CB neighborhood of the identity, and is generated by a CB subset. Then the identity map is a quasi-isometry between $G$ endowed with the word metrics from any two symmetric, CB generating sets. 
\end{theorem} 

Thus, word metrics from CB generating sets can be used to {\em define} the quasi-isometry type of a locally CB group.  
However, even determining whether any given group has a well-defined large-scale structure is a challenging question -- analogous to whether a particular group admits a finite generating set.   We show that, among the big mapping class groups, there is a rich family of examples to which Rosendal's theory applies, and give the first steps towards a QI classification of such groups.

% - - - - - - - - - - - - - - - - - - - - - - - 
\subsection{Main results} 
For simplicity, we assume all surfaces are oriented and have empty boundary, and all homeomorphisms are orientation-preserving.  (The cases of non-orientable surfaces, and those with finitely many boundary components can be approached using essentially the same tools.)

\subsection*{Summary} We give a complete classification of surfaces $\Sigma$ for which $\mcg(\Sigma)$ is {\em locally CB} (Theorem \ref{thm:loc_CB}), a complete classification under mild hypotheses of those which are additionally {\em CB generated}, so have a well-defined quasi-isometry type (Thoerem \ref{thm:CB_generated}), and which are {\em globally CB}, i.e. have trivial QI type (Theorem \ref{thm:global_CB}).    

To give the precise statements, we need to recall the classification of surfaces and state two key definitions.  

\subsection*{End spaces} By a theorem of Richards \cite{Richards} orientable, boundaryless,
infinite type surfaces are completely classified by the following data:  the genus (possibly infinite), the {\em space of ends} $E$, which is a totally disconnected, separable, metrizable topological space, and the subset of ends $E^G$ that are accumulated by genus, which is a closed subset of $E$.   Every such pair $(E, E^G)$ occurs as the space of ends of some surface, with $E^G = \emptyset$ iff the surface has finite genus.
We call a pair $(E, E^G)$ {\em self-similar} if for any decomposition $E = E_1 \sqcup E_2 \sqcup \ldots \sqcup E_n$ of $E$ into pairwise disjoint clopen sets, there exists a clopen set $D$ contained in some $E_i$ such that the pair $(D, D\cap E^G)$ is homeomorphic to $(E, E^G)$.   

\subsection*{Complexity} A key tool in our classification is the following ranking of the ``local complexity'' of an end, which (as we show) gives a partial order on equivalence classes of ends.  
\begin{definition} \label{def:partial_order} 
For $x, y \in E$, we say $x \lesseq y$ if every neighborhood of $y$ contains a homeomorphic copy of a neighborhood of $x$. We say $x$ and $y$ are {\em equivalent} if  $x \lesseq y$ and $y \lesseq x$.  
\end{definition} 
We show that this order has maximal elements (Proposition \ref{prop:maximal_element}), and for $A$ a clopen subset of $E$, we denote the maximal ends of $A$ by $\calM(A)$. 

The following theorem gives the classification of locally CB mapping class groups.  While the statement is technical, it is easy to apply in specific examples.  For instance, the surfaces in Figure \ref{fig:locCB} (left) satisfy the conditions, while those on the right fail to have CB mapping class group.  
  \begin{figure*}[h]
     \centerline{ \mbox{
 \includegraphics[width = 4.5in]{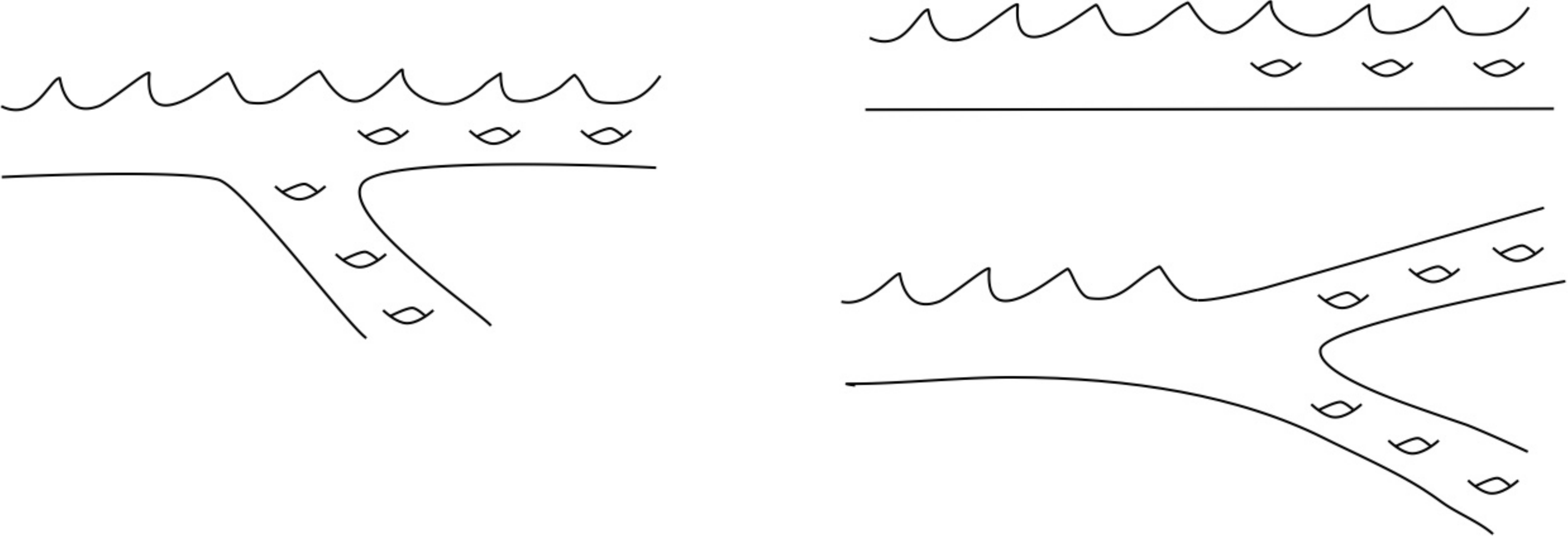}}}
 \caption{By Theorem \ref{thm:loc_CB}, the surface on the left has a locally CB mapping class group and those on the right do not.  All have $\calP = \emptyset$.}
  \label{fig:locCB}
  \end{figure*}

\begin{theorem}[\textbf{Classification of locally CB mapping class groups}] \label{thm:loc_CB}
$\mcg(\Sigma)$ is locally CB if and only if %either $\Sigma$ has 0 or infinite genus with self-similar end space (in which case it is globally CB), or 
there is a finite type surface $K \subset \Sigma$ such that the complimentary regions of $K$ each have infinite type and 0 or infinite genus, and partition $E$ into finitely many clopen sets 
$$E = \left( \bigsqcup_{A \in \calA} A \right) \sqcup \left( \bigsqcup_{P \in \mathcal{P}} P \right)$$
such that: 
\begin{enumerate}[ parsep=0pt, partopsep=2pt, itemsep=0pt]
\item Each $A \in \calA$ is self-similar, with $\calM(A) \subset \calM(E)$ and $\calM(E) \subset \sqcup_{A \in \calA}\calM(A)$, 
\item each $P \in \mathcal{P}$ is homeomorphic to a clopen subset of some $A  \in \calA$, and 
\item for any $x_A \in \calM(A)$, and any neighborhood $V$ of the end $x_A$ in $\Sigma$, there is $f_V \in \Homeo(\Sigma)$ so that $f_V(V)$ contains the complimentary region to $K$ with end set $A$.
\end{enumerate}
Moreover, in this case $\calV_K := \{g \in \Homeo(\Sigma) : g|_K = id \}$ is a CB neighborhood of the identity. 
\end{theorem} 

In order to illustrate Theorem \ref{thm:loc_CB} and motivate the conditions in the next two classification theorems, we now state results in the much simpler special case when $\Sigma$ has genus zero and countable end space.  

\subsection*{Special case: $E$ countable, genus zero}
If $E$ is a countable set and $E^G = \emptyset$, a classical result of Mazurkiewicz and Sierpinski \cite{MS} states that there exists a countable ordinal $\alpha$ such that $E$ is homeomorphic to the ordinal $\omega^\alpha n +1$ equipped with the order topology.   
Thus, any $x \in E$ is locally homeomorphic to $\omega^\beta+1$ for some $\beta \leq \alpha$ (here $\beta$ is the Cantor-Bendixon rank of the point $x$). In this case, our partial order $\less$ agrees with the usual ordering of the ordinal numbers, points are equivalent iff they are locally homeomorphic, and we have the following.  

\begin{theorem}[Special case of Theorems \ref{thm:loc_CB}, \ref{thm:CB_generated}, and \ref{thm:global_CB}]  \label{thm:countableE}
Suppose $\Sigma$ is an infinite type surface of genus 0 with $E \cong \omega^\alpha n + 1$, then  
\begin{enumerate}[label=\roman*), parsep=0pt, partopsep=2pt, itemsep=0pt]
\item $\mcg(\Sigma)$ is CB if and only if $n =1$; in this case $E$ is self-similar.  
\item If $n\geq 2$ and $\alpha$ is a successor ordinal, then $\mcg(\Sigma)$ is locally CB and generated by a CB set, but admits a surjective homomorphism to $\Z$, so is not globally CB.  
\item If $n \geq 2$ and $\alpha$ is a limit ordinal, then $\mcg(\Sigma)$ is locally CB, but not generated by any CB set.  
\end{enumerate}
\end{theorem}

\subsection*{Classification: general case}
One cannot hope for such a clean statement as Theorem \ref{thm:countableE} to hold in general, since there is no similarly clean classification of end spaces.  In fact, even in the genus zero case, classifying possible end spaces $E$ (i.e. closed subsets of Cantor sets) up to homeomorphism is a difficult and well studied problem, equivalent to the classification problem for countable Boolean algebras.\footnote{By Stone duality, totally disconnected, separable, compact sets are in 1-1 correspondence with countable Boolean algebras.}   Ketonen \cite{Ketonen} gives some description and isomorphism invariants; but in practice, these are difficult to use, despite being in a sense an optimal classification, as Carmelo and Gao show in \cite{CamerloGao} that the isomorphism relation is Borel complete.    Our definition of the partial order $\lesseq$ allows us to sidestep the worst of these issues.

For technical reasons, the order is better behaved under a weak hypothesis on the topology of the end space which we call ``tameness."  See Section \ref{sec:CBgen} for motivation and the definition. 
To our knowledge, tame surfaces include all concrete examples studied thus far in the literature, including the mapping class 
groups of some specific infinite type surfaces in 
 \cite{Bavard,APV, FHV}, and the discussion of geometric or dynamical properties of various translation surfaces of infinite type in \cite{Chaman,Hooper,Rand}.  
Although non-tame examples do exist (see Example~\ref{ex:non_tame}) 
there are no known non-tame surface that have a well defined quasi-isometry 
type (Problem \ref{prob:non_tame}).   Under this hypothesis, we can give a complete classification of surfaces with a well defined QI type, and those with a trivial QI type as follows.

\begin{theorem}[\textbf{Classification of CB generated mapping class groups}] \label{thm:CB_generated}
For a tame surface $\Sigma$ with locally (but not globally) CB mapping class group, $\mcg(\Sigma)$ is CB generated if and only if $E$ is {\em finite rank} and not of {\em limit type}.  
\end{theorem}

\begin{theorem}[\textbf{Classification of globally CB mapping class groups}] \label{thm:global_CB}
Suppose $\Sigma$ is either tame or has countable end space.  Then $\mcg(\Sigma)$ is CB if and only if $\Sigma$ has infinite or zero genus and $E$ is self-similar or a variant of this called ``telescoping".  The telescoping case occurs only when $E$ is uncountable.  
\end{theorem}

{\em Finite rank}, loosely speaking, means that finite index subgroups of $\mcg(\Sigma)$ do not admit surjective homomorphisms to $\Z^n$ for arbitrarily large $n$.   {\em Limit type} refers to behavior of equivalence classes for the partial order that mimics the behavior of limit ordinals in the special countable case stated above.  See Section \ref{sec:limit-type}. 
{\em Telescoping} is a slightly broader notion of homogeneity or local similarity of an end space.   Informally speaking, {\em self-similar} sets either appear very homogeneous (e.g. a Cantor set) or may have one ``special" point, any neighborhood of which contains a copy of the whole set -- for instance, a countable set with a single accumulation point is self-similar.  Telescoping is a generalization that allows for two special points.  See section \ref{sec:telescoping} for further motivation and the precise definition.

\subsection*{Key tool:  Nondisplaceable subsurfaces}  
The following tool is of independent interest and provides an easily employable criterion to certify that a surface has non-CB mapping class group (or, equivalently, admits a continuous isometric action on a metric space with unbounded orbits).  

\begin{definition}A connected, finite type subsurface $S$ of a surface $\Sigma$ is called 
{\em nondisplaceable} if $f(S) \cap S \neq \emptyset$ for each $f \in \Homeo(\Sigma)$.  
A non-connected surface is nondisplaceable if, for every $f \in \Homeo(\Sigma)$ there are connected components $S_i$, $S_j$ of $S$ such that $f(S_i) \cap S_j \neq \emptyset$.  \end{definition}

\begin{theorem}  \label{thm:nondisplace}
If $\Sigma$ is a surface that contains a nondisplaceable finite type subsurface, then 
$\mcg(\Sigma)$ is not globally CB.   
\end{theorem}
A key ingredient of the proof is {\em subsurface projection}, a familiar tool from the study of mapping class groups of finite type surfaces introduced by Masur and Minsky \cite{MM2}.  

Theorem \ref{thm:nondisplace} immediately gives many examples of surfaces whose mapping class groups are not CB, hence 
admit unbounded orbits on combinatorial complexes.  For instance, any surface with finite but nonzero genus has this property.  (See Theorem \ref{thm:countableE} below for a number of other easily described examples.)
Theorem~\ref{thm:nondisplace} also recovers, with a new proof, some of the work of Bavard in \cite{Bavard} and Durham-Fanoni-Vlamis in \cite{DFV}.  

\subsection*{Outline} 
\begin{itemize}[label=$\bullet$, parsep=0pt, partopsep=2pt, itemsep=0pt]
\item Section \ref{sec:nondisplace} contains background information on standard mapping class group techniques and the proof of Theorem \ref{thm:nondisplace}. 
\item Section \ref{sec:SS_tele_CB} gives two criteria for CB mapping class groups: self-similarity and telescoping end spaces.  This is used later in the proof of the local and global CB classification theorems. 
\item Section \ref{sec:order} introduces the partial order on the end space and proves key properties of this relation, and a characterization of self-similar end spaces in terms of the partial order. 
\item Section \ref{sec:loc_CB} contains the proof of Theorem \ref{thm:loc_CB}.  This and the following section form the technical core of this work. 
\item Section \ref{sec:CBgen} contains the proof of Theorem \ref{thm:CB_generated}.  
\item Section \ref{sec:global_CB} gives the proof of Theorem \ref{thm:global_CB}.  
\end{itemize} 

\begin{acknowledge}
K.M. was partially supported by NSF grant  DMS-1844516.   
K.R was partially supported by NSERC Discovery grant RGPIN 06486.
Part of this work was completed at the 2019 AIM workshop on surfaces of infinite type. 
We thank Camille Horbez and Justin Lanier for helpful comments on an earlier version of this paper. 
\end{acknowledge}

%- - - - - - - - - - - - - - - - - - - - - - - - - - - - - - - - - -
\section{Proof of Theorem \ref{thm:nondisplace}}  \label{sec:nondisplace}
In this section we prove that nondisplaceable finite type subsurfaces of a surface $\Sigma$ are responsible for nontrivial geometry in $\mcg(\Sigma)$.  We begin by introducing some notions from large-scale geometry and setting some conventions that will be useful throughout.   

\paragraph{A criterion for coarse boundedness.}
Recall that a subset $A \subset G$ of a metrizable, topological group is said to be {\em coarsely bounded} or {\em CB} if it has finite diameter in every compatible left-invariant metric on $G$.  
The following result gives an equivalent condition that is often easier to use in practice.  

\begin{theorem}[Rosendal, Theorem 1.10 in \cite{RosendalLargeScale}]  
\label{thm:CB_criterion} 
Let $A$ be a subset of a separable, metrizable group $G$.  The following are equivalent
\begin{enumerate}
\item $A$ is coarsely bounded.
\item For every neighborhood $\calV$ of the identity in $G$, there is a finite subset $\calF$
 and some $k \geq 1$ such that $A \subset (\calF \calV)^k$
\end{enumerate} 
\end{theorem} 

While Rosendal's theory is quite broadly applicable, mapping class groups (of any manifold) fall into the nicest family to which it applies, namely the completely metrizable or {\em Polish} groups.  
 For any manifold $M$, the homeomorphism group $\Homeo(M)$ endowed with the compact-open topology is Polish, and hence also for any closed subset of $M$, the closed subgroups $\Homeo(M, X)$ and $\Homeo(M \, \mathrm{rel} X)$ of homeomorphisms respectively preserving and pointwise fixing $X$.  (In the mapping class groups context, $X$ is typically taken to be the boundary of $M$ or a set of marked points.)
Thus, since the identity component $\Homeo_0(M,X)$ is a closed, normal subgroup, the quotient $\Homeo(M, X)/\Homeo_0(M,X)$ is also a Polish group.\footnote{For the case where $M$ is a surface, that mapping class groups are Polish was also observed in \cite{APV} using the property that these groups are the automorphism groups of the curve complex of the surface.}

One useful tool for probing the geometry of a topological group is the following concept of a length function.  

\begin{definition}
A {\em length function} on a topological group $G$ is a continuous function $\ell: G \to [0, \infty)$ satisfying  $\ell(g) = \ell(g^{-1})$, $\ell(id) = 0$, and $\ell(gh) \leq \ell(g) + \ell(h)$ for all $g, h \in G$.  
\end{definition} 
If $\ell$ is any length function, then for 
any $\epsilon > 0$ the set $\ell^{-1}([0,\epsilon))$ is a neighborhood of the identity in $G$.  It 
follows from the criterion in Theorem \ref{thm:CB_criterion} that $\ell$ is bounded on any 
CB subset. 

Our strategy for the proof of Theorem \ref{thm:nondisplace} is to use the presence of a nondisplaceable subsurface to construct an unbounded length function.  
In order to do this, we introduce some notation and conventions which will also be used in later sections.

\paragraph{Surfaces: conventions.}
The following conventions will be used throughout this work.  Infinite type surfaces, typically denoted 
by $\Sigma$, are assumed to be connected and orientable, and unless otherwise specified will be assumed to have empty boundary. 
By a {\em curve} in $\Sigma$ we mean a free homotopy
class of a non-trivial, non-peripheral, simple closed curve. 
In the first part of this section, when we talk about a 
subsurface $S \subset \Sigma$, we always assume that $S$ is connected, has finite type 
and is {\em essential} meaning that every curve in $\partial S$ is non-trivial and 
non-peripheral in $\Sigma$.  (Later we will broaden our discussion to include non-connected subsurfaces.)
As is standard, the {\em complexity} of a finite type surface $S$ is defined to be $\xi(S) = 3 g_S + b_S + p_S$
where $g_S$ is the genus, $p_S$ is the number of punctures and $b_S$
is the number of boundary components of $S$.   {\em Finite type} simply means that all these quantities are finite.

The {\em intersection number} between two curves $\gamma_1$
and $\gamma_2$, is the usual geometric intersection number $\I(\gamma_1, \gamma_2)$ defined to be the minimal 
intersection number between representatives in the free homotopy classes of $\gamma_1$ and $\gamma_2$. 
To simplify the exposition going forward, we will fix a complete hyperbolic structure on $\Sigma$. 
Then every curve has a unique geodesic representative and the homotopy class
of every subsurface has a unique representative that has geodesic boundary.  A pair of curves $\gamma_1$ and $\gamma_2$ have disjoint representatives iff their geodesic representatives are disjoint. In this case, we say that $\I(\gamma_1, \gamma_2) = 0$. 
Otherwise, we say $\gamma_1$ intersects $\gamma_2$ and in this case, the intersection number 
$\I(\gamma_1, \gamma_2)$ is the cardinality of the intersection of their geodesic representatives.  

Similarly, two subsurfaces $R$ 
and $S$ (or a subsurface $R$ and geodesic $\gamma$) intersect if every subsurface homotopic to $R$ intersects every subsurface homotopic 
to $S$ (or analogously for $\gamma$), and this is equivalent to saying that the representatives of $R$ and $S$ with 
geodesic boundaries intersect each other. 
Hence, from now on, every time we consider a curve we assume it is a geodesic and
every time we consider a subsurface we assume it has geodesic 
boundary.  This allows us to unambiguously speak of intersections.

\begin{definition} \label{def:non-displaccable} 
A finite type, connected subsurface $S \subset \Sigma$ is 
{\em nondisplaceable} if $S \cap f(S) \neq \emptyset$ for all $f \in \mcg(\Sigma)$.
\end{definition}  

\begin{example} \label{ex:finite-genus}
When $\Sigma$ has positive, finite genus any subsurface $S$ whose genus matches that of $\Sigma$ is nondisplaceable.  This is because $S$ contains non-separating curves but $\Sigma -S$ does not.  Since every image of $S$ under a homeomorphism of $\Sigma$ will also contain a non-separating curve, it must intersect $S$. 
\end{example} 

\begin{example}[Nondisplaceable subsurfaces] \label{ex:easy_nondisplace}
It is also easy to construct examples of nondisplaceable surfaces using the topology of the end space.   If $Z$ is any invariant, finite set of ends of cardinality at least 3, then any surface $S$ which separates points of $Z$ into different complimentary regions will be nondisplaceable. 

Similarly, if $X$ and $Y$ are disjoint, closed invariant sets of ends, with $X$ homeomorphic to a Cantor set, then a subsurface homeomorphic to a pair of pants which contains points of $X$ in two complementary regions, and all of $Y$ in the third complimentary region will also be nondisplaceable.  
\end{example} 

\paragraph{Curve graphs and subsurface projections.} 
We recall some basic material on curve graphs.  A reader unfamiliar with this machinery may wish to consult the introductory notes \cite{Schleimer} or paper \cite{MM1} for more details.   As in the previous paragraph, we continue to assume here that surfaces are connected.  

The {\em curve graph} $\CS$ of a surface $S$ is a graph whose vertices are curves 
in $S$ and whose edges are pairs of disjoint curves.  We give each edge length one 
and denote the induced metric on $\CS$ by $d_S$. With this metric, as soon as $\xi(S) \geq 5$, 
$(\CS, d_S)$ has infinite dimeter and is Gromov hyperbolic \cite{MM1}.   One can define curve graphs analogously for infinite type surfaces,  but we will use only the classical finite type setting.  

If $\Sigma$ is any surface and $S \subset \Sigma$ a subsurface there is a {\em projection map} $\pi_S$ from the set of curves
in $\Sigma$ that intersect $S$ to the set of subsets of $\CS$, defined as follows: for a curve 
$\gamma$, the intersection $\gamma \cap S$ of the geodesic $\gamma$ with the 
subsurface $S$ is either equal to $\gamma$ (if $\gamma \subset S$) or is a union of arcs 
with end points in $\partial S$. For every such arc $\omega$, 
one may perform a surgery between $\gamma$ and $\partial S$ to obtain in curve
in $S$ disjoint from $\omega$, possibly in two different ways (the curve is a 
concatenation of one or two copies of $\omega$ and one or two arcs in $\partial S$). 
We define the projection $\pi_S(\gamma)$ to be $\gamma$ if $\gamma \subset S$ 
and otherwise to be the the {\em union} of curves associated to each arc on $\gamma \cap S$
obtained by surgery as above. When $\xi(S) \geq 5$, the set $\pi_S(\gamma)$ has diameter at most 
$2$ in $\CS$, in fact, we have 
\begin{equation} \label{eq:projection-lipschitz} 
\I(\gamma_1,\gamma_2)=0 \quad \Longrightarrow \quad 
\diam_S \pi_S(\gamma_1 \cup \gamma_2) \leq 2. 
\end{equation}
(see \cite[Lemma 2.2]{MM2} for more details). In general, if $\mu$ is a subset of $\CS$, 
we define 
\[
\pi_S(\mu) = \bigcup_{\gamma \in \mu} \pi_S(\gamma). 
\]

The natural distance $d_S$ on $\CS$ can be extended to a distance function on curves 
in $\Sigma$ that intersect $S$ via 
\[
d_S(\gamma_1, \gamma_2) = \min_{\alpha_i \in  \pi_S(\gamma_i)} d_S(\alpha_1, \alpha_2).
\]
The following result states that a bound on the intersection number between two 
curves gives a bound on their projection distance in any subsurface.  This principle 
is well known and there are many similar results in the literature. We give a short proof 
with a suboptimal bound.  

\begin{lemma}
Let $\gamma_1$ and $\gamma_2$ be curves in $\Sigma$ that intersect $S$. Then 
\begin{equation} \label{Eq:int>distance}
d_S(\gamma_1, \gamma_2) \leq 2 \log_2\big( \I(\gamma_1, \gamma_2)+1\big) + 6
\end{equation}
\end{lemma}

\begin{proof}
Let $\omega_1$ be an arc in $S$ that is a component of the restriction of $\gamma_1$ 
and let $\alpha_1 \in \pi_S(\gamma_1)$ be the curve in $\CS$ that is obtained by doing 
a surgery between $\omega_1$ and the boundary of $S$. Then $\alpha_1$ is a 
concatenation of 
one or two copies of $\omega_1$ (depending on whether the end points of $\omega_1$
are on the same boundary or different boundary components of $S$) and
some arcs in $\partial S$. Similarly, let $\omega_2$ be an arc in $S$ that is a restriction 
of $\gamma_2$ and $\alpha_2$ be the associated curve in $\pi_S(\gamma_2)$.
Then every intersection point between $\omega_1$ and $\omega_2$ results in 
$1$, $2$ or $4$ intersection points between $\alpha_1$ and $\alpha_2$. Also, 
applying surgery between $\omega_2$ and $\partial S$ can result in two intersection
points between $\alpha_2$ and $\alpha_1$ at each end of $\omega_2$. Therefore, 
\[
\I(\alpha_1, \alpha_2) \leq 4 \I(\omega_1, \omega_2) + 4,
\]
On the other hand, from \cite[Lemma 1.21]{Schleimer}, we have 
\[
d_S(\alpha_1, \alpha_2) \leq 2 \log_2(\I(\alpha_1, \alpha_2)) + 2.
\]
Therefore, 
\begin{align*}
d_S(\alpha_1, \alpha_2) &\leq  2 \log_2(4 \I(\omega_1, \omega_2)+4) + 2\\
&\leq 2 \log_2( \I(\gamma_1, \gamma_2)+1) + 6
\end{align*} 
which is as we claimed. 
\end{proof}
 
The notions of distance $d_S$ and intersection number can also be extended further 
to take finite sets of curves as arguments.   If $\mu_i$ are finite sets of curves, we define
\[
d_S(\mu_1, \mu_2) = \max_{\gamma_1 \in \mu_1, \gamma_2 \in \mu_s}
d_S(\gamma_1, \gamma_2)
\qquad\text{and}\qquad
\I(\mu_1, \mu_2) = \max_{\gamma_1 \in \mu_1, \gamma_2 \in \mu_s}
\I(\gamma_1, \gamma_2).
\]
Using Equation~\eqref{Eq:int>distance}, for any finite subsets $\mu_1$ and $\mu_1$
of $\CS$, we have
\begin{equation} \label{Eq:marking}
d_S(\mu_1, \mu_2) \leq 2 \log_2( \I(\mu_1, \mu_2)+1) + 6\\
\end{equation} 
Note that the triangle inequality still holds for this generalized distance $d_S$.  

\paragraph{Construction of an unbounded length function.} 
We now proceed with the proof of Theorem \ref{thm:nondisplace}.  
Let $\Sigma$ be any surface, and let $S$ be a  nondisplaceable subsurface.  Enlarge $S$ if needed so that $\xi(S) \geq 5$ and so that $S$ is connected.  (In Section \ref{sec:disconnected}, we give an alternative modification for non-connected subsurfaces that will be useful in later work.)

Let $\mathcal{I}$ denote the set of (isotopy classes of) subsurfaces of the same type as $S$, i.e. 
\[\mathcal{I} = \big\{ f(S) \mid f \in \mcg(\Sigma)\big\}. \]  
As usual, while $f(S)$ denotes only an isotopy class of a surface when $f \in \mcg(\Sigma)$, the reader may identify it with an honest subsurface by taking the representative with geodesic boundary.  
Let $\mu_S$ be a filling set of curves in $\CS$, i.e. a set of curves with the property that every curve in $S$ intersects some
curve in $\mu$. 

For $R \in \mathcal{I}$ let $\mu_R = \pi_R(\mu_S)$. Note that this is alway defined since
$\mu_S$ fills $S$ and $R$ intersects $S$ because $S$ was assumed nondisplaceable.  

Now, define 
\[
\ell \from \mcg(\Sigma) \to \Z \text{ by } 
\ell(\phi) = \max_{R \in \mathcal{I}}  d_{ \phi(R)} \big( \phi( \mu_R) , \mu_{ \phi (R) } \big) . 
\]

Equivalently, we have $\ell(\phi) = \max \limits_{T \in \mathcal{I}} d_T \big( \phi( \mu_{\phi^{-1}(T)}) , \mu_T \big)$.  

Note that $\ell$ is finite because, for every $\phi$, the intersection number $\I\big(\mu_S, \phi (\mu_S) \big)$
is a finite number. Hence, by Equation~\eqref{Eq:marking}, 
their projections to $\phi(R)$ lie at a bounded distance in $\calC(R)$, with a bound
that depends on $\phi$ alone, not on $R$. 

The latter definition also makes it clear that $\ell(\phi) = \ell(\phi^{-1})$, since 
\begin{align*} 
\ell(\phi^{-1}) &= \max_{T \in \mathcal{I}} d_T \big( \phi^{-1}( \mu_{\phi(T)}) , \mu_T \big) \\
&= \max_{T \in \mathcal{I}} d_{\phi(T)} \big(  \mu_{\phi(T)} , \phi(\mu_T) \big) \\
&= \max_{R = \phi(T) \in \mathcal{I}} d_{R} \big(  \mu_{R} , \phi(\mu_{\phi^{-1}(R)}) \big) 
= \ell(\phi) .
\end{align*}

We now check the triangle inequality.  Let $\psi$ and $\phi$ be given, and let $R \in \mathcal{I}$ be a surface such that $\ell (\psi \phi) = d_{ \psi \phi(R)} \big(\psi \phi( \mu_R) , \mu_{ \psi \phi (R) } \big)$.  Then we have
\begin{align*}
\ell (\psi \phi) &= d_{ \psi \phi(R)} \big(\psi \phi( \mu_R) , \mu_{ \psi \phi (R) } \big) \\
 & \leq  d_{ \psi \phi (R)} \big(\psi \phi(  \mu_R) , \psi ( \mu_{ \phi (R) } ) \big) 
   +  d_{ \psi \phi (R)} \big(\psi ( \mu_{ \phi (R)} ) , \mu_{ \psi \phi (R) } \big) \\
 & = d_{ \phi (R)} \big( \phi(  \mu_R) , \mu_{ \phi (R) }  \big) 
   +  d_{ \psi (Q)} \big(\psi ( \mu_Q ) , \mu_{ \psi (Q)} \big), 
       \qquad \text{ where $Q = \phi (R)$ }\\
  & \leq  \ell(\phi) + \ell(\psi). 
\end{align*} 

Continuity of $\ell$ as a function on $\mcg(\Sigma)$ is immediate, since for any given $\phi \in \mcg(\Sigma)$, the preimage of $\ell(\phi)$ under $\ell$ contains the open set consisting of mapping classes $\phi'$ so that $\phi(\mu_S) = \phi'(\mu_S)$.  Note also that $\ell(id) = 0$.  This proves that $\ell$ is a length function.  

To see that $\ell$ is unbounded, let $\phi \in \mcg(\Sigma)$
be a homeomorphism that preserves $S$ and such that the restriction $\phi|_S$ of $\phi$
to $S$ is a pseudo-Anosov homeomorphism of $S$. Then for any curve 
$\gamma$ in $S$, 
\begin{equation} \label{eq:unbounded}
d_S\big(\gamma, \phi^n(\gamma)\big) \to \infty
\qquad\text{as}\qquad n \to \infty. 
\end{equation}
(see e.g. \cite{MM1} for details). 
Thus, $\ell$ is an unbounded length function, and so $\mcg(\Sigma)$ is not coarsely bounded.  \qed

\subsection{Disconnected subsurfaces} \label{sec:disconnected}
While we have so far worked only with connected nondisplaceable subsurfaces, there is a natural generalization of the work above to non-connected subsurfaces.  
This will be useful when we need to find a non-displacable subsurface that is disjoint from 
a given compact subset of $\Sigma$ to determine if $\mcg(\Sigma)$ is locally CB. 
The extension to this broader framework requires a little care, since, if we simply take the definitions above verbatim then the diameter of the curve graph $C(S)$ is finite as soon as $S$ is not connected.  However, the following minor adaptations allow our work above to carry through in this case.

\begin{definition} \label{def:disconnected} 
A {\em disconnected finite type
subsurface} is a disjoint union of 
$\overline{S}=\{S_1, \dots, S_k\}$ of disjoint finite type
subsurfaces. 
Such a 
subsurface $\overline S$ is {\em nondisplaceable}, 
for any $f \in \mcg(\Sigma)$ and any component $S_i \in \overline{S}$,
there is $S_j \in \overline{S}$ such that $S_j \cap f(S_i) \neq \emptyset$. 
\end{definition}  

We now use $\overline{S}$ to construct a length function on $\mcg(\Sigma)$. As before, 
let $\mathcal{I}$ denote the set of images of $\overline{S}$ under mapping classes, i.e. 
\[\mathcal{I} = \Big\{ f(\overline{S}) \mid f \in \mcg(\Sigma)\Big\}. \] 
An element $\overline{R}$ of $\mathcal{I}$ is simply the disjoint union of a set $\{R_1, \ldots R_k \}$ where $R_i = f(S_i)$.  
Let $\mu_{\overline S}$ be a set of curves in $\cup_i \mathcal{C}(S_i)$ 
that fill every $S_i$. 
Keeping the notation from before, note that $\pi_{R_i}(\mu_{\overline S})$ is always defined since $R_i$ intersects some 
$S_j$ and curves in $\mu_{\overline S}$ fill $S_j$. Now, define 
\[
\ell_{\overline S} \from \mcg(\Sigma) \to \Z, \quad 
\ell_{\overline S}(\phi) = \max_{\overline{R} \in \mathcal{I}} \, \max \Big\{  d_{ \phi(R_i)} \big( \phi(\mu_{R_i}) , \mu_{ \phi (R) } \big) \mid R_i \text{ a component of } \overline{R} \Big\}
\]
The same computation as in the connected case shows that $\ell_{\overline S}$ 
is finite, is continuous as a function on $\mcg(\Sigma)$, and satisfies the triangle 
inequality. To see that $\ell_{\overline S}$ is unbounded, let $\phi \in \mcg(\Sigma)$
be a homeomorphism that preserves $\overline{S}$ and such that the restriction 
$\phi|_{S_1}$ of $\phi$ to $S_1$ is a pseudo-Anosov homeomorphism of $S_1$. 
Since $\ell_{\overline S}$ is defined as a maximum of distances in various curve graphs, 
if $\phi$ has a positive translation length in $\mathcal{C}(S_1)$ (or in any 
$\mathcal{C}(S_i)$) then $\ell_{\overline S}(\phi^n) \to \infty$ as $n \to \infty$.   This gives an alternative proof of Theorem \ref{thm:nondisplace} in the disconnected case, and the following more general statement.

\begin{proposition} \label{prop:disconnected}
If $\Sigma$ contains a connected or disconnected, nondisplaceable, finite type subsurface $S$ such that each connected component of $S$ has complexity at least $5$, then there exists a length function $\ell$ defined on $\mcg(\Sigma)$ such that the restriction of $\ell$ to mapping classes supported on $S$ is unbounded.    
\end{proposition}

\section{Self-similar and telescoping end spaces} \label{sec:SS_tele_CB}
In this section we give two topological conditions (in Propositions \ref{prop:SSCB} and \ref{prop:telescoping}) that imply coarse boundedness of the mapping class group: {\em self-similarity} and {\em telescoping}.   

\subsection{Self-similar end spaces} 
Recall that a space of ends $(E, E^G)$ is said to be {\em self-similar} if for any decomposition $E = E_1 \sqcup E_2 \sqcup \ldots \sqcup E_n$ of $E$ into pairwise disjoint clopen sets, there exists a clopen set $D$ in some $E_i$ such that $(D, D\cap E^G)$ is homeomorphic to $(E, E^G)$.  
There are many examples of such sets, a few basic ones are:
\begin{itemize} [label=$\bullet$, parsep=0pt, partopsep=2pt, itemsep=0pt]
\item $E$ equal to a Cantor set, and $E^G$ either empty, equal to $E$, or a singleton.
\item $E$ a countable set homeomorphic to $\omega^\alpha +1$ with the order topology, 
for some countable ordinal $\alpha$, $E^G$ is the set of points of type 
$\omega^\beta+1$ for all ordinals $\beta \geq \beta_0$ where $\beta_0$ is a some
fixed ordinal. 
\item $E$ the union of a countable set $Q$ and a Cantor set where the sole accumulation 
point of $Q$ is a point in the Cantor set, and $E^G = \overline{Q}$. 
\end{itemize}

\begin{convention} Going forward, we drop the notation $E^G$, assuming that $E$ comes with a designated closed subset of ends accumulated by genus, empty if the genus of $\Sigma$ is finite, and that all homeomorphisms between sets or subsets of end spaces preserve (setwise) the ends accumulated by genus.
\end{convention}

Since $E$ and $E^G$ are totally disconnected spaces, we also make the following convention.  

\begin{convention} 
For the remainder of this work, when we speak of a {\em neighborhood} in an end space $E$, we always mean a {\em clopen neighborhood}.  
\end{convention}

\begin{proposition}[Self-similar implies CB] \label{prop:SSCB}
Let $\Sigma$ be a surface of infinite or zero genus.  If the space of ends of $\Sigma$ is self-similar, then $\mcg(\Sigma)$ is CB. 
\end{proposition} 
Note that finite, nonzero genus surfaces cannot have CB mapping class group by Example \ref{ex:finite-genus}, so Proposition \ref{prop:SSCB} is optimal in this sense.   Note also that the Proposition holds for finite type surfaces as well, the only applicable example is the once-punctured sphere which has trivial mapping class group. 

\begin{proof}[Proof of Proposition \ref{prop:SSCB}]
Let $\Sigma$ be an infinite type surface satisfying the hypotheses of the proposition, and let $\calV$ be a neighborhood of the identity in $\mcg(\Sigma)$.   Then there exists some finite type subsurface $S$ such that $\calV$ contains the open set $\calV_S$ consisting of mapping classes of homeomorphisms that restrict to the identity on $S$.  
 By Theorem \ref{thm:CB_criterion}, it suffices to find a finite set 
$\calF \subset \mcg(\Sigma)$ and $k \in \N$ (which are allowed to depend on $\calV_S$, 
hence on $S$) such that $\mcg(\Sigma) = (\calF\calV_S)^k$.  
Enlarging $S$ (and therefore shrinking $\calV_S$) if needed, we may assume that each connected component of $\Sigma-S$ is of infinite type.  

The connected components of $\Sigma - S$, together with the finite set $P_S$
of punctures of $S$, partition $E$ into clopen sets.  Let 
\[
E =  E_S^0 \sqcup E_S^1 \sqcup \ldots \sqcup E_S^n \sqcup P_S
\] 
denote this decomposition, and let $\Sigma_S^i$ denote the connected component of $\Sigma - S$ containing $E_S^i$.  Since $S$ is of finite type, $E^G \cap P_S= \emptyset$.   Since $E$ is self-similar, one of the $E^i_S$ contains a copy of $E$.  Without loss of generality, we assume this is $E_S^0$, the set of ends of $\Sigma_S^0$.  
The next lemma asserts that we may find a surface $R$ as depicted in Figure \ref{fig:self-sim}. 

  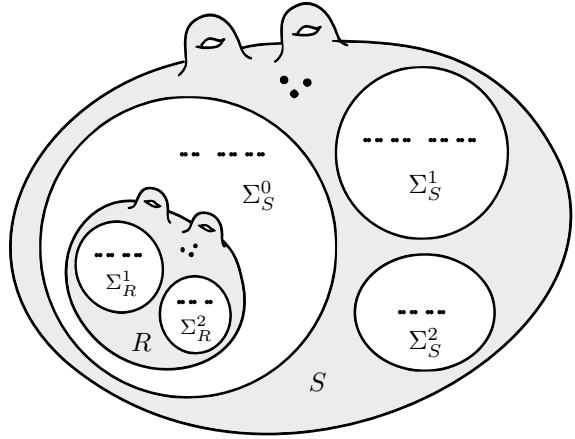
\begin{figure*}[h]
   \labellist 
  \small\hair 2pt
   \pinlabel $S$ at 160 30
   \pinlabel $R$ at 70 50 
   \pinlabel $\Sigma_S^0$ at 130 125 
   \pinlabel $\Sigma_S^1$ at 215 130 
   \pinlabel $\Sigma_S^2$ at 215 50 
   \pinlabel {\scriptsize $\Sigma_R^2$} at 98 57 
   \pinlabel {\scriptsize $\Sigma_R^1$} at 60 80
   \endlabellist
     \centerline{ \mbox{
 \includegraphics[width = 3in]{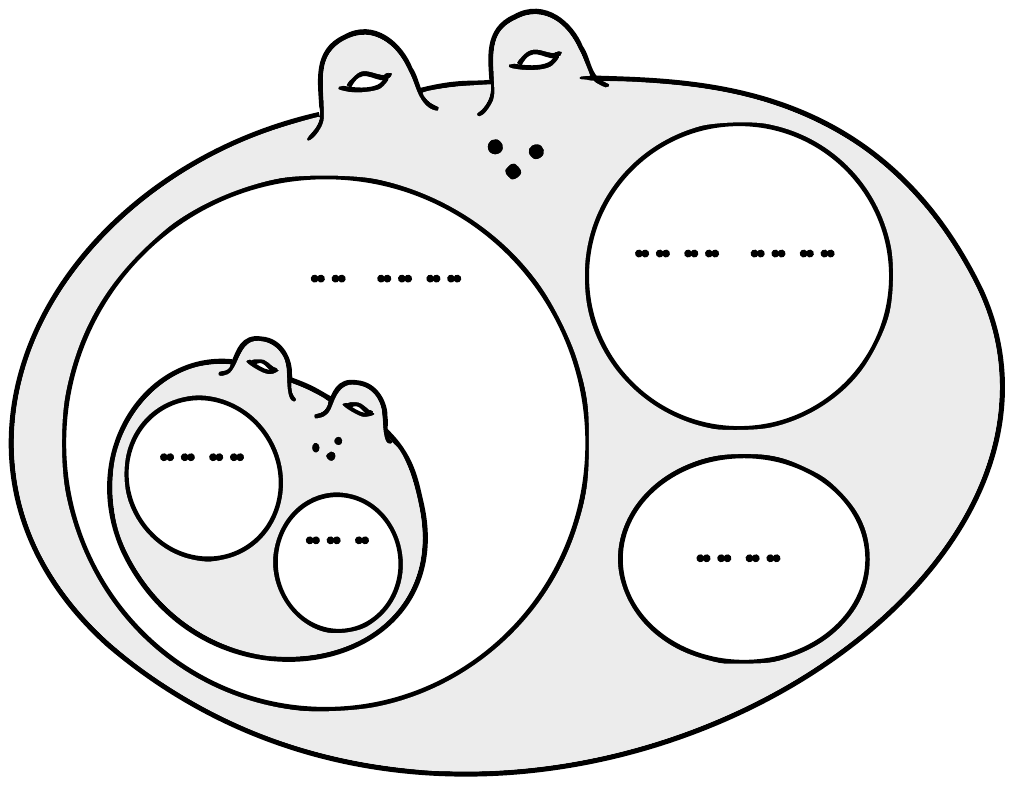}}}
 \caption{A homeomorphic copy $R$ of $S$ contained in the complimentary region $\Sigma_S^0$.}
  \label{fig:self-sim}
  \end{figure*}

\begin{lemma}  \label{lem:find_R}
Let $\Sigma_S^0$ be a connected component of $\Sigma - S$ whose end set contains a homeomorphic copy $E'$ of $E$.  Then $\Sigma_S^0$ contains a homeomorphic copy $R$ of $S$ such that there exists a homeomorphism $f_{SR}$ of $\Sigma$ with $f_{SR}(S) = R$ and $f_{SR}(\Sigma_S^0) \supset S$.
\end{lemma}

\begin{proof} 
Since $E'$ is homeomorphic to $E$, we may find inside $E'$ pairwise disjoint clopen sets $E^i_R$, $i=1, 2, \ldots, n$ and $P_R$ with $E^i_R$ homeomorphic to $E^i_S$ and $P_R$ homeomorphic to $P_S$. (As always, we mean via homeomorphisms which respect $E^G$.)    Richards' classification of surfaces implies that there is a homeomorphism $f$ of $\Sigma$ such that, for $i \geq 1$, $f(E^i_S) = E^i_R$, 
$f(P_S) = P_R$ and that restricts to the identity on the complement of 
$\left( \bigcup_{i \geq 1} E^i_R \cup E^i_S \right) \cup P_R \cup P_S$.  In fact, one can take such a homeomorphism to move $S$ to a surface $R$ disjoint from $S$, with the properties of $f_{SR}$ claimed in the Lemma.  
In detail, take connected neighborhoods $\Sigma^i_R$ of $E^i_R$, pairwise disjoint from each other and from $\bigcup_{i\geq 1} E^i_S \cup P_S \cup P_R$, each with a single boundary component.   Let $c$ be a curve in $\Sigma_S^0$ separating $\bigcup_{i \geq 1} E^i_R \cup P_R$ from the other ends, and (shrinking one of the surfaces $\Sigma^i_R$ to exclude some genus from it if needed) so that the connected component of 
$(\Sigma - \bigcup_{i \geq 1} \Sigma^i_R) - c$
containing $P_R$ has the same finite genus as $\Sigma$.   Then the union of $c$ and the boundary components of the $\Sigma^i_R$ bound a surface homeomorphic to $S$, with complimentary regions homeomorphic to the complementary regions of $S$, and the classification of surfaces tells us that we may find a homeomorphism $f_{SR}$ with $f_{SR}(\partial \Sigma_S^0) = c$, $f_{SR}(\Sigma^i_S) = \Sigma^i_R$, and $f_{SR}(S) = R$.  By construction, this also satisfies $f_{SR}(\Sigma_S^0) \supset S$.  
\end{proof}

Now fix $R$ and $f_{SR}$ as in Lemma \ref{lem:find_R} and let 
$\mathcal{F} = \{f_{SR}, f^{-1}_{SR}\}$. We will show 
\[\mcg(\Sigma) = (F \calV_S)^5.\] 
Let $g \in \mcg(\Sigma)$.   Let $E'$ be a homeomorphic copy of $E$ in $\Sigma^0_S$, and consider the set $g(E')$.  
Since the clopen sets $Z := (E_R^0 \cap E_S^0)$,  $(E_S^0 - Z)$ and $(E_R^0 -Z)$ partition $E$, their intersections with $g(E')$ partition $g(E')$.  Since $g(E')\cong E$ is a self-similar set, one of these three sets in the partition contains a homeomorphic copy of $E$; call this $E''$.  
Thus, $E''$ lies either in $g(E') \cap E_R^0$ or in $g(E') \cap E_S^0$ (or both).  
If the first case occurs, then we have $f_{SR}^{-1} g(E') \cap f_{SR}^{-1}(E_R^0) = f_{SR}^{-1} g(E') \cap E^0_S$.  This means, that, 
at the cost of replacing $g$ by $f_{SR}^{-1} g$, and therefore using one more letter from $F$, we can assume that we are in the second case, i.e. where $E'' \subset g(E') \cap E_S^0$.  So it suffices to show that in this case, we have $g \in (\mathcal{F} \calV_S)^4$.  This situation is illustrated in Figure \ref{fig:g(S)}.  

  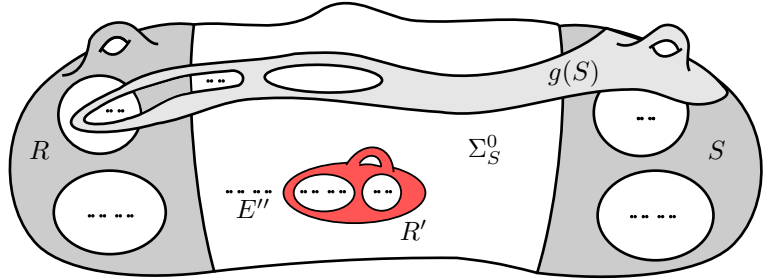
\begin{figure*}[h]
   \labellist 
  \small\hair 2pt
   \pinlabel $R'$ at 230 30
   \pinlabel $R$ at 20 75 
   \pinlabel $g(S)$ at 320 118 
   \pinlabel $\Sigma_S^0$ at 270 75
   \pinlabel $S$ at 400 75 
   \pinlabel $E''$ at 140 42
   \endlabellist
     \centerline{ \mbox{
 \includegraphics[width = 4in]{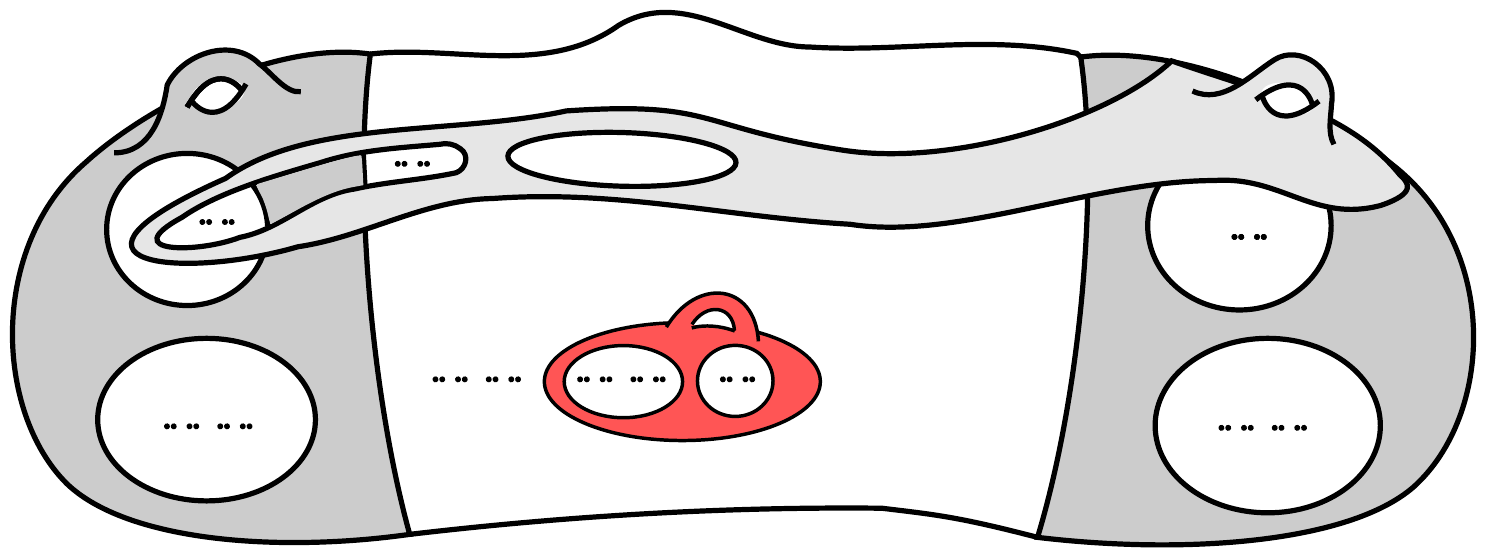}}}
 \caption{$g(S)$ may intersect $R$ and $S$ in a complicated way, but $R'$ lies in the ``big'' complimentary region of at least one of them (here it is in their intersection, $Z$)}
  \label{fig:g(S)}
  \end{figure*}

Assuming that $E'' \subset g(E') \cap E_S^0$, we can now find another copy of $R$ in a small neighborhood of $E''$, and hence in $g(\Sigma_S^0) \cap \Sigma_S^0$.    Specifically, just as in Lemma \ref{lem:find_R}, but using $E''$ instead of $E'$, and working with the subsurface $R$ of the surface $\Sigma_S^0$ instead of the subsurface $S$ of $\Sigma$,  we may find a surface $R' \subset \Sigma_S^0 \cap g(\Sigma_S^0)$ homeomorphic to $R$, and a homeomorphism $v$ of $\Sigma_S^0$ mapping $R$ to $R'$ that satisfies $v(\Sigma_R^0 \cap \Sigma_S^0) \supset R$.  
Extend $v$ to a homeomorphism of $\Sigma$ by declaring it to be the identity on $\Sigma - \Sigma_S^0$.  Abusing notation slightly, denote this homeomorphism also by $v$, and so we have $v \in \calV_S$.  

By construction, we also have $g(S) \cup S \subset v(\Sigma_R^0)$.  
The same argument to that in Lemma \ref{lem:find_R} using the classification of surfaces now shows that we may find $u$ restricting to the identity on $R’$ with $ug(S) = S$ and $ug$ equal to identity on $S$.   
(The details as a straightforward exercise.)
Since $u$ is the identity on $R'$, it follows that $(v f_{SR})^{-1} u (v f_{SR})$ is the identity on $(v f_{SR})^{-1}(R') = S$, which implies that $u \in (\mathcal{F} \calV_S)^3$, hence $g \in (\mathcal{F} \calV_S)^4$.  
This concludes the proof of Proposition \ref{prop:SSCB}. 
\end{proof}

\subsection{Telescoping end spaces}  \label{sec:telescoping}

\textbf{Motivation.} Recall from Example \ref{ex:easy_nondisplace} that, if $\Sigma$ is a surface such that there exists a finite, $\mcg(\Sigma)$-invariant set $F \subset E$ of cardinality at least 3, then $\mcg(\Sigma)$ is not CB: any finite-type subsurface $S$ such that the elements of $F$ each lie in different complimentary regions of $S$ is easily seen to be non-displaceable.  
The definition of {\em telescoping} below was motivated by the question: {\em under what conditions is a 2-element $\mcg(\Sigma)$-invariant subset of $E$ compatible with global coarse boundedness?}
As will follow from our work in Section \ref{sec:global_CB}, this never happens if $E$ is 
countable: every surface with countable end space and coarsely bounded mapping class group is self-similar.  However, in the uncountable case, surfaces with telescoping end spaces provide additional examples (and are the only additional examples among tame surfaces).  
Informally speaking, telescoping spaces of ends have two ``special'' points with the property that neighborhoods of each point can be expanded an arbitrary amount, and can also be expanded a fixed amount relative to a neighborhood of the other point. 

\begin{convention}
In the following definition, and for the remainder of this work, a {\em neighborhood of an end} $x$ in $\Sigma$ means a connected subsurface with a single boundary component that has $x$ as an end.     
\end{convention}

\begin{definition}  \label{def:telescoping}
A surface $\Sigma$ is {\em telescoping} if there are ends
$x_1, x_2 \in E$ and disjoint clopen neighborhoods $V_i$ of $x_i$ in $\Sigma$ such that 
for all clopen neighborhoods $W_i \subset V_i$ of $x_i$, there exist homeomorphisms $f_i, h_i$ of $\Sigma$, both pointwise fixing $\{x_1, x_2\}$, with 
\[ 
f_i(W_i) \supset (\Sigma - V_{3-i}) \qquad
h_i(W_i) = V_i, \qquad \text{ and }\qquad h_i(V_{3-i}) = V_{3-i}. 
\]
When we wish to make the points $x_1, x_2$ explicit, we say also {\em telescoping with respect to \{$x_1, x_2$\}}. 
\end{definition} 
Note that this definition implies that $\Sigma$ has infinite or zero genus, as does $\Sigma - (V_1 \cup V_2)$.  

While the compliment of a Cantor set in $S^2$ is both self-similar and telescoping with respect to any pair of points, there are many examples of telescoping sets that are {\em not} self-similar, for instance:
\begin{itemize}
\item $E^G$ a Cantor set, $E$ the union of  $E^G$ and another Cantor set which intersects $E^G$ at exactly two points
\item $E$ the union of two copies of the Cantor set, $C_1$ and $C_2$, which intersect at exactly two points, and a countable set $Q$ such that the accumulation points of $Q$ are exactly $C_1$.  $E^G$ could be empty, equal to the closure of $Q$, or equal to $E$.  
%(See figure \ref{fig:telescoping})
\end{itemize} 
In the second example above, note that $E^G$ cannot be, for instance, a singleton as depicted in Figure \ref{fig:telescoping} there are homeomorphisms of $E$ that satisfy the requirements of $h_i$ and $f_i$ above, but these do not come from homeomorphisms of $\Sigma$ that have the required behavior on neighborhoods of the ends, due to genus considerations.  

  \begin{figure*}[h]
     \labellist 
  \small\hair 2pt
   \pinlabel $x_1$ at -5 20 
   \pinlabel $x_2$ at 195 20 
   \endlabellist
     \centerline{ \mbox{
 \includegraphics[width = 3in]{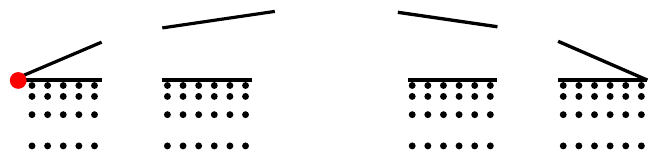}}}
 \caption{A surface that narrowly fails the definition of telescoping.  $E^G = \{x_1\}$ is the singleton shown in red. If $E^G$ were empty, the complement of this set in $S^2$ would be a telescoping surface}
  \label{fig:telescoping}
  \end{figure*}
\noindent Of course, there are many more complicated telescoping sets; including numerous variations on the second example above.  

\begin{remark} \label{rk:for_all_small}
Note that an equivalent definition of telescoping may be given by replacing ``there exist disjoint neighborhoods $V_i$ of $x_i$" with ``for all sufficiently small neighborhoods $V_i$ of $x_i$." 
The proof is an immediate consequence of the definition. 
\end{remark}

\begin{proposition}[Telescoping implies CB] \label{prop:telescoping}
Let $\Sigma$ be a surface that is telescoping with respect to $\{x_1, x_2\}$, then the pointwise stabilizer of $\{x_1, x_2\}$ in $\mcg(\Sigma)$ is CB.  
\end{proposition} 

In particular, if $\{x_1, x_2\}$ is a $\mcg(\Sigma)$-invariant set, then $\mcg(\Sigma)$ is itself CB.  

\begin{proof} 
Suppose that $\Sigma$ is telescoping and let $x_i,V_i$ be as in the definition.
To simplify notation, let $G$ denote the pointwise stabilizer of $\{ x_1, x_2\}$ in $\mcg(\Sigma)$. 
Fix a neighborhood of the identity in $\mcg(\Sigma)$, shrinking this if needed we may take it to be the set  $\calV_S$ of mapping classes that restrict to the identity on some finite type subsurfaces $S$.    By Remark \ref{rk:for_all_small}, we may assume that $S \subset \Sigma - (V_1 \cup V_2)$.   
Let $\calV \subset \calV_S$ be the set of mapping classes that restrict to the identity on $\Sigma' := \Sigma - (V_1 \cup V_2)$.
We will exhibit a finite set $\calF$ such that $G \subset (\calF \calV)^10 \subset (\calF \calV_S)^10$.  This is sufficient to show that $G$ is CB by Theorem \ref{thm:CB_criterion}.

Fix neighborhoods $W_i \subset V_i$ of $x_i$ in $E$ and homeomorphisms $f_i$ with $f_i(W_i) \supset (E - V_{3-i})$ as given by the definition of telescoping.   Let $\calF = \{ f_1^{\pm 1}, f_2^{\pm 1} \}$.   This is our finite set.   
Note that any homeomorphism which restricts to the identity on $V_2$ lies in $f_1 \calV f_1^{-1}$.  

Given $g \in G$, let $W_i'$ be a neighborhood of $x_i$ such that $W_i' \subset g(V_i) \cap V_i$.   Then, by definition of telescoping, there exists a homeomorphism $h_1$ restricting to the identity on $V_2$ and taking $W_i'$ to $V_1$; as well as another, say $j_1$, taking $W_1$ to $V_1$. Thus $g_1 := j_1^{-1} h_1$ is the identity on $V_2$, hence lies in $f_1 \calV f_1^{-1}$, and satisfies $g_1(W_1') = W_1$. 

Similarly, we can find $g_2 \in f_2 \calV f_2^{-1}$ restricting to the identity on $V_1$ with $g_2(W'_2) = W_2$.  Thus, 
\[ g_2 g_1 g (W_i') = W_i \text{ for } i = 1,2. \]
It follows that $g_2 g_1 g(\Sigma') \subset (\Sigma - W_1 \cup W_2)$,  and so $f_1^{-1} g_2 g_1 g(\Sigma') \subset V_2$ and
\[f_1^{-2} g_2 g_1 g(\Sigma') \subset W_2.\]
For notational convenience, let $\phi = f_1^2 g_2 g_1 g$.  Since $\phi(\Sigma')$ and $f_1^{-1}f_2\Sigma'$ both lie in $W_2$, as a consequence of the definition of telescoping there exists a homeomorphism $\psi$ restricting to the identity on $V_1$, with $\psi \phi (\Sigma') = f_1^{-1}f_2\Sigma'$.  Precomposing $\psi$ with a homeomorphism that is also the identity on $V_1$, we can also ensure that $(f_1^{-1}f_2)^{-1}\psi \phi$ restricts to the identity on $\Sigma'$.  
Thus, we have shown that $f_1 f_2^{-1} \psi \phi g \in \calV$, hence $g \in (\calF \calV)^10$.   
Note that this exponent could be decresed at the cost of enlarging $\calF$ to include words of length two in $\{f_1^{\pm1}, f_2^{\pm1}\}$.
\end{proof}

We conclude this section with a result whose proof serves as a good warm-up for the technical work to come.  

\begin{proposition} \label{prop:telescoping_uncountable}
No telescoping surface has countable end space.    
\end{proposition}

\begin{proof}
Suppose that $\Sigma$ has countable end space $E$, so $E \cong \omega^\alpha n +1$ by \cite{MS}, and $E^G$ is some closed subset.  
Assume for contradiction that $E$ is telescoping with respect to some pair of ends $x_1$, $x_2$.   
For each point $x \in E$, there exists $\beta= \beta(x) \leq \alpha$ such that every sufficiently small neighborhood of $x$ is homeomorphic to $\omega^\beta +1$ (this ordinal $\beta(x)$ is simply the Cantor-Bendixon rank of $x$).   It follows from the definition of telescoping that every clopen neighborhood $U$ of $x_i$ disjoint from $x_{3-i}$ is homeomorphic to each other such neighborhood.  In particular, necessarily $n=2$ and $x_1$ and $x_2$ are points of equal and maximal rank $\alpha$.  Suppose as a first case that $\alpha$ is a sucessor ordinal and let $\eta$ denote its predecessor.  Then the set of points of rank $\eta$ accumulates only at $x_1$ and $x_2$.  If $V_i$ are any neighborhoods of $x_i$, then $\Sigma - (V_1 \cup V_2)$ contains finitely many points of rank $\eta$.  Thus, if $W_1 \subset V_1$ satisfies that $V_1 - W_1$ contains exactly one point of rank $\eta$, then no homeomorphism fixing $V_2$ can send $W_1$ to $V_1$, and the definition of telescoping fails. 

The case where $\alpha$ has limit type is similar.  Given neighborhoods $V_i$ of $x_i$, let $\eta < \alpha$ be the supremum of the ranks of points in $E - (V_1 \cup V_2)$.  Let $W_1 \subset V_1$ be a set such that $V_1 - W_1$ contains a point of rank $\alpha$ where $\eta < \alpha$.  Then no homeomorphism fixing $V_2$ can send $W_1$ to $V_1$, and the definition of telescoping fails. 
\end{proof} 

As we will see in the next sections, this limit type phenomenon is closely related to the failure of the mapping class group to be generated by a CB set.  However, to treat this also in the case where $E$ is uncountable, we will need to develop a more refined ordering on the space of ends.

\section{A partial order on the space of ends}  \label{sec:order}

Let $\Sigma$ be an infinite type surface with set of ends $(E, E^G)$. 
As in the previous section, we drop the notation $E^G$ and, by convention, all homeomorphisms of an end space $E$ of a surface $\Sigma$ 
are required to preserve $E^G$, so to say that $A \subset E$ 
is homeomorphic to $B \subset E$ means that there is a homeomorphism from 
$(A, A \cap E^G)$ to $(B, B \cap E^G)$.   
It follows from Richards' classification of surfaces in \cite[Theorem 1]{Richards} that  each homeomorphism of $(E, E^G)$ is 
induced by a homeomorphism of $\Sigma$.\footnote{While this is not in the statement of Theorem 1 of \cite{Richards}, the proof gives such a construction.  This was originally explained to the authors by J. Lanier following work of  S. Afton.}
Thus, we will pass freely between speaking of 
homeomorphisms of the end space and the underlying surface. 

Observe also that, if $U$ and $V$ are two disjoint, clopen subsets of $E$, then any homeomorphism $f$ from $U$ onto $V$ can be extended to a globally defined homeomorphism $\bar{f}$ of $E$ by declaring $\bar{f}$ to agree with $f^{-1}$ on $V$ and to pointwise fix the complement of $U \cup V$.  Thus, to say points $x$ and $y$ are locally homeomorphic is equivalent to the condition that there exists $\bar{f} \in \mcg(\Sigma)$ with $\bar{f}(x) = y$.  We will use this fact frequently.  In particular, we have the following equivalent rephrasing of Definition \ref{def:partial_order}. 

\begin{definition} 
Let $\lesseq$ be the binary relation on $E$ where $y \lesseq x$ if, for every
neighborhood $U$ of $x$, there exists a neighborhood $V$ of $y$ and 
$f \in \mcg(\Sigma)$ so that $f(V) \subset U$. 
\end{definition} 

Note that this relation is transitive.  

\begin{notation}
For $x, y \in E$ we say that $x \sim y$ or ``$x$ and $y$ are of the same type'' if $x \lesseq y$ and $y \lesseq x$, and write $E(x)$ for the set $\{y \mid y \sim x\}$. 
\end{notation}

It is easily verified that $\sim$ defines an equivalence relation: symmetry and reflexivity are immediate from the definition, while transitivity follows from the transitivity of $\lesseq$.  
From this it follows that the relation $\less$, defined by $x \less y$ if $x \lesseq y$ and $x \nsim y$, gives a partial order on the set of equivalence classes under $\sim$.   
Note that, for any homeomorphism $f$ of $\Sigma$, we have $x \succ y$ 
(respectively, $x \succcurlyeq y$) if and only of $f(x) \succ f(y)$
(respectively, $f(x) \succcurlyeq f(y)$). 

\begin{proposition} \label{prop:countable_nice}
If $E$ is countable, then $x \sim y$ if and only if $x$ and $y$ are locally homeomorphic.   If additionally $E^G = \emptyset$, then the Cantor-Bendixon rank gives an order isomorphism between equivalence classes of points and countable ordinals. 
\end{proposition} 

\begin{proof}
Suppose that $E$ is countable, so every point $x \in E$ has a neighborhood $U_x$ homeomorphic (in this proof alone we break convention briefly and do not require homeomorphisms to respect $E^G$) to the set $\omega^{\alpha(x)} +1 $ where $\alpha(x)$ is the Cantor-Bendixon rank of $x$.  If $x \lesseq y$ and $y \lesseq x$ both hold, it follows that $\alpha(x) = \alpha(y)$, and so any homeomorphism from a  neighborhood of $x$ into a neighborhood of $y$ necessarily takes $x$ to $y$.  In particular, these points also have the same rank.  
\end{proof} 

\begin{remark}
We do not know if Proposition \ref{prop:countable_nice} holds in the uncountable case.  This appears to be an interesting question.  However, it is quite easy to construct large families of examples for which it does hold. 
\end{remark} 

\begin{remark}
Despite the above remark, there are indeed some marked differences between the behavior of $\less$ when $E$ is countable and when $E$ is uncountable. 
In the countable case, it follows from Proposition \ref{prop:countable_nice} that $x \less y$ if and only if $y$ is an accumulation point of $E(x)$, giving a convenient alternative description of $\less$.   In general, if $y$ is an accumulation point of $E(x)$, then $x \lesseq y$, as proved below.  However,  if $E$ is a Cantor set and $E^G = \emptyset$, for example,  then all points are equivalent and all are accumulation points of their equivalence class.  
\end{remark} 

We now prove some general results on the structure of $\lesseq$.  

\begin{lemma} \label{lem:closed}
For every $y \in E$, the set $\{x \mid x \moreq y\}$ is closed.  
\end{lemma}

\begin{proof}
Consider a a sequence $x_n \to x$ where $x_n \succcurlyeq y$ holds for all $n$.   
Let $U$ be a neighborhood 
of $x$.  Then, for large $n$, $U$ is also a neighborhood of $x_n$ and hence 
contains homeomorphic copies of some neighborhood of $y$. 
\end{proof}

\begin{proposition} \label{prop:maximal_element} 
The partial order $\succ$ has maximal elements. Furthermore, for every
maximal element $x$, the equivalence class $E(x)$ is either finite or a Cantor set. 
\end{proposition}

\begin{proof}
To show that $E$ has maximal elements, by Zorn's lemma, it suffices to show 
that every chain has an upper bound. Suppose that $\mathcal{C}$ is a totally ordered 
chain.  Consider the family of sets $\{ x  \mid x \succcurlyeq y \}$, for $y \in \mathcal{C}$. 
Then, by Lemma \ref{lem:closed}, this is a family of nested, closed, non-empty 
sets and hence 
\[
\mathcal{C}_M = \bigcap_{y \in \mathcal{C}} \big\{ x  \mid x \succcurlyeq y \big\}
\]
is non-empty.  By definition, any point of this intersection is an upper bound for 
$\mathcal{C}$.

To see the second assertion, consider a maximal element $x$. If $E(x)$ is an infinite 
set, then it has an accumulation point, say $z$.  Then $z \moreq x$, but since $x$ is maximal, we have 
$z \sim x$.  Since any neighborhood of any other point in $E(x)$ contains a homeomorphic copy of a neighborhood of $z$, it follows that all points of $E(x)$ are accumulation points and hence $E(x)$
is a Cantor set. 
\end{proof}

Going forward, we let $\mathcal{M}= \mathcal{M}(E)$ denote the set of maximal elements for $\succ$. 

%----------------------------------------
\subsection{Characterizing self-similar end sets}
The remainder of this section consists of a detailed study of the behavior of end sets using the partial order.  We will develop a number of tools for the classification of locally CB and CB generated 
mapping class groups that will be carried out in the next sections. 

\begin{proposition} \label{prop:M_self_sim}
Let $\Sigma$ be a surface with end space $E$ and no nondisplaceable subsurfaces.  
Then $E$ is self-similar if and only if $\calM$ is either a singleton or a Cantor set of points of the same type  
\end{proposition} 

One direction is easy and does not require the assumption that $\Sigma$ has no nondisplaceable subsurfaces: if $\calM$ contains two distinct maximal types $x_1$ and $x_2$, then a partition $E = E_1 \sqcup E_2$ where $E(x_i) \subset E_i$ fails the condition of self-similarity.  Similarly, if $\calM$ is a finite set of cardinality at least two, then any partition separating points of $\calM$ similarly fails the condition.   By Proposition \ref{prop:maximal_element}, the only remaining possibility is that $\calM$ is a Cantor set of points of the same type.  This proves the first direction.    The converse is more involved, so we treat the singleton and Cantor set case separately.  We will need the following easy observation.  

\begin{observation}[``Shift maps''] \label{obs:shift}
Suppose that $U_1, U_2, \ldots$ are disjoint, pairwise homeomorphic clopen sets which Hausdorff converge to a point.  Then $\bigcup_{i=1}^\infty U_i$ is homeomorphic to $\bigcup_{i=2}^\infty U_i$
\end{observation}

\begin{proof}
For each $i$, fix a homeomorphism $f_i: U_i \to U_{i+1}$.  Since the $U_i$ Hausdorff converge to a point, the union of these defines a global homeomorphism $\bigcup_{i=1}^\infty U_i \to \bigcup_{i=2}^\infty U_i$.
\end{proof}

We will also use the following alternative characterization of self-similarity.  
\begin{lemma} \label{lem:two_sets_suffice}
Self-similarity is equivalent to the following condition:  If $E = E_1 \sqcup E_2$ is a decomposition into clopen sets, then some $E_i$ contains a clopen set homeomorphic to $E$.  
\end{lemma} 

\begin{proof}
Self-similarly implies the condition by taking $n=2$.  For the converse, suppose the condition holds and let $E = E_1 \sqcup E_2 \sqcup \ldots \sqcup E_n$ be a decomposition into clopen sets.   Grouping these as $E_1 \sqcup (E_2 \sqcup \ldots \sqcup E_n)$, by assumption one of these subsets contains a clopen set $E'$ homeomorphic to $E$.  If it is $E_1$, we are done.  Else, the sets 
$E' \cap E_i$ $(i = 2, 3, ... n)$ form a decomposition of $E' \cong E$ into clopen sets; so by the same reasoning either $E_2 \cap E'$ contains a clopen set homeomorphic to $E$, or the union of the sets $E' \cap E_i$, for $i \geq 3$ does.   Iterating this argument eventually produces a set homeomorphic to $E$ in one of the $E_i$.  
\end{proof} 

The next three lemmas give the proof of Proposition \ref{prop:M_self_sim}.  

\begin{lemma} \label{lem:move_into}
Suppose $\Sigma$ has no nondisplaceable subsurfaces and $\calM$ is a singleton.  Then for any decomposition 
 $E = A \sqcup B$, if $\calM \subset A$ then $A$ contains a homeomorphic copy of $B$.  
\end{lemma}

\begin{proof}
Let $E = A \sqcup B$ be a decomposition of $E$ into clopen sets with $\calM = \{x\} \subset A$.    Since $A$ is a neighborhood of $x$, every point $y \in B$ has a neighborhood homeomorphic to a subset of $A$.   Since $B$ is compact, finitely many of these cover, say $U_1, U_2, \ldots, U_k$.  Without loss of generality, we may assume all the $U_i$ are disjoint and their union is $B$.  Let $V_i$ be a homeomorphic copy of $U_i$ in $A$, note that $x \notin \bigcup_i V_i$.    Let $S$ be a three-holed sphere subsurface so that the disjoint sets $\{x\}$, $\bigcup_i V_i$ and $B$ all lie in different connected components of the complement of $S$.  
Let $f$ be a homeomorphism displacing $S$.  Since $f(x) = x$, up to replacing $f$ with its inverse, we have either $f(B) \subset A$, in which case we are done, or that $A$ contains a homeomorphic copy of $A \sqcup \left( \bigcup_i V_i \right)$.  In this latter case, by iterating $f$ we can find $k$ disjoint copies of $\bigcup_i V_i$ inside $A$.  Since each contains a copy of $U_i$, this gives a subset of $A$ homeomorphic to $\sqcup U_i = B$.  
\end{proof} 

As a consequence, we can prove the first case of Proposition \ref{prop:M_self_sim}
\begin{lemma} \label{lem:M_singleton}
Suppose $\Sigma$ has no nondisplaceable subsurfaces and $\calM$ is a singleton.  Then $E$ is self-similar.  
\end{lemma}

\begin{proof}
Let $E = E_1 \sqcup E_2$ be a decomposition of $E$ into clopen sets.  Without loss of generality, suppose $\calM = \{x\} \subset E_1$.  Lemma \ref{lem:move_into} says that there is a homeomorphic copy $U_2$ of $E_2$ inside $E_1$, necessarily this is disjoint from $\{x\}$.  Let $A$ be a smaller neighborhood of $x$, disjoint from $U_2$.  Lemma \ref{lem:move_into} again gives a homeormorphic copy $U_3$ of $E_2$ inside $A$.  Proceeding in this way, we may find $E_2 = U_1, U_2, U_3, \ldots$ each homeomorphic to $E_2$ and Hausdorff converging to $x$, then apply Observation \ref{obs:shift}. 
\end{proof}

The second case is covered by the following.  
\begin{lemma} \label{lem:M_Cantor}
Suppose $\Sigma$ has no nondisplaceable subsurfaces and $\calM$ is a Cantor set of points all of the same type.  Then $E$ is self-similar.  
\end{lemma}

\begin{proof} 
Let $E = E_1 \sqcup E_2$ be a decomposition of $E$ into clopen sets.  If $\calM$ is contained in only one of the $E_i$, then one may apply the argument from Lemma~\ref{lem:M_singleton}, by letting $x$ be any point of $\calM$.  Thus, we assume that both $E_1$ and $E_2$ contain points of $\calM$. 

For concreteness, fix a metric on $E$.  For each $n \in N$, fix a decomposition $A^{(n)}_1, \ldots A^{(n)}_{j_n}$ of $E$ into clopen sets of diameter at most $2^{-n}$, such that $E_1$ and $E_2$ are each the union of some number of these sets.   Let $S_n$ be a subsurface homeomorphic to a $j_n$-holed sphere, with complimentary regions containing the sets $A^{(n)}_k$.  Since $S_n$ is displaceable, there exists some $k$ such that $A^{(n)}_k$ contains a copy of all but one of the sets $A^{(n)}_j$, in particular it contains either $E_1$ or $E_2$.   Passing to a subsequence, this gives us some $i$ such that there exist homeomorphic copies of $E_i$ of diameter less than $2^n$, for each $n$.  Without loss of generality, say that this holds for $E_1$.  Passing to a further subsequence, we can assume these copies of $E_1$ Hausdorff converge to a point $x$, so in particular every neighborhood of $x$ contains a copy of $E_1$. 

It follows from the definition of $\lesseq$ that each $y \in \calM$ therefore also has this property: every neighborhood of $y$ contains a homeomorphic copy of $E_1$.  Let $y_2, y_3, y_4, \ldots$ be a sequence of points in $E_2$ converging to $y \in E_2$, let $U_1 = E_1$.  Fix disjoint neighborhoods $N_i$ of $y_i$ converging to $y$, and let $U_i$ be a homeomorphic copy of $E_1$ in $N_i$.  Now apply Observation \ref{obs:shift}.  
\end{proof} 

This completes the proof of Proposition \ref{prop:M_self_sim}

\subsection{Stable neighborhoods}
Motivated by the behavior of maximal points in the Proposition above, we make the following definition.  

\begin{definition}  \label{def:stable_nbhd}
For $x \in E$, call a neighborhood $U$ of $x$ {\em stable} if for any smaller neighborhood $U' \subset U$ of $x$, there is a homeomorphic copy of $U$ contained in $U'$.  
\end{definition} 
Our use of the terminology ``stable" is justified by Lemma \ref{lem:stable_nbhd} below, 
which says that all such neighborhoods of a point are homeomorphic.  
(Recall that, by convention, neighborhood always means clopen neighborhood.)

\begin{remark} \label{rem:stable_nbhd}
Note that stable neighborhoods are automatically self-similar sets, and that if $U$ is a 
stable neighborhood of $x$, then $x \in \calM(U)$.  Our work in the previous section shows  
that when $\less$ has a unique maximal type and all subsurfaces are displaceable, each 
maximal point has a stable neighborhood.  
\end{remark}

It follows immediately from the definition that if $x$ has one stable neighborhood, then {\em every} sufficiently small neighborhood of $x$ is also stable.  More generally, we have the following.  

\begin{lemma} \label{lem:stable_easy}
If $x$ has a stable neighborhood, and $y \sim x$, then $y$ has a stable neighborhood. 
\end{lemma} 

\begin{proof} 
Let $U$ be a stable neighborhood of $x$.  Since $y \less x$, there is a neighborhood $V$ 
of $y$ such that $U$ contains a homeomorphic copy of $V$.  Suppose $V' \subset V$ is a
smaller neighborhood of $V$.  Since $x \less y$, there is some neighborhood $U'$ of $x$ (without loss of generality, we may assume that $U' \subset U$) such that $V'$ contains a homeomorphic copy of $U'$.  By definition of stable neighborhoods, $U'$ contains a 
homeomorphic copy of $U$, thus $V'$ contains a homeomorphic copy of $U$ and hence a homeomorphic copy of $V$.  
\end{proof}

\begin{lemma} \label{lem:stable_nbhd}
If $x$ has a stable neighborhood $U$, then for any $y \sim x$, all sufficiently small neighborhoods of $y$ are homeomorphic to $U$ via a homeomorphism taking $x$ to $y$.  
\end{lemma}  

\begin{proof}
The proof is a standard back-and-forth argument.  
Suppose $x \less y$ and $y \less x$.  Let $V_x$ be a stable neighborhood of $x$ and $V_y$ a stable neighborhood of $y$.  
Take a neighborhood basis $V_x = V_1 \supset V_2 \supset V_3 \ldots$ of $x$ consisting of nested neighborhoods, and take a neighborhood basis $V_y = V'_1 \supset V'_2 \supset V'_3 \ldots$ of $y$.  
Since $y \less x$ and $x \less y$, each $V_i$ contains a homeomorphic copy of $V'_0$ and each $V'_i$ a copy of $V_0$.  

Let $f_1$ be a homeomorphism from $V_1 - V_2$ into $V'_0$.  Note that we may assume the image of $f_1$ avoids $y$ -- if $y$ is the unique maximal point of $V'_0$, then this is automatic, otherwise, $E(y)$ is a Cantor set of points, each of which contains copies of $V_0$ in every small neighborhood.   
Let $g_1$ be a homeomorphism from the complement of the image of $f_1$ in $V'_1 - V'_2$ onto a subset 
of $V_2 - \{x\}$.  Iteratively, define $f_i$ to be a homeomorphism from the complement of the 
image of $g_{i-1}$ in $V_{i}-V_{i+1}$ onto a subset of $V'_i - \{y\}$, and $g_i$ a homeomorphism 
from the complement of the image of $f_i$ in $V'_i - V'_{i+1}$ onto a subset of $V_{i+1} - \{x\}$.  
Then the union of all $f_i$ and $g_i^{-1}$ is a homeomorphism from $V_1 - \{x\}$ to 
$V'_1 - \{y\}$ that extends to a homeomorphism from $V_1$ to $V_1'$ taking $x$ to $y$.  
\end{proof}

The following variation on Lemma \ref{lem:move_into} uses stable neighborhoods as a replacement for displaceable subsurfaces.  

\begin{lemma} \label{lem:consume} 
Let $x, y \in E$, and assume $x$ has a stable neighborhood $V_x$, and that $x$ is an accumulation point of $E(y)$.  Then for any sufficiently small neighborhood $U$ of $y$, $U \cup V_x$ is homeomorphic to $V_x$. 
\end{lemma}

\begin{proof}
If $x \sim y$, then let $U$ be a stable neighborhood of $y$ disjoint from $V_x$.  
Let $V_1 \supset V_2 \supset V_3 \ldots$ be a neighborhood basis for $x$ consisting of stable neighborhoods.  
Since $x$ is an accumulation point of $E(y)$, for any sufficiently small neighborhood 
$U_0$ of $y$ (and hence for any stable neighborhood $U$), there is a homeomorphic 
copy $U_1$ of $U_0$ in $V_1-\{x\}$. Shrinking neighborhoods if needed, we may take $U_1$ 
to be disjoint from $V_{i_1}$ for some $i_1 \in \N$.  Since $V_{i_1}$ is homeomorphic to 
$V_1$, there is also a homeomorphic copy of $U_2$ of $U_0$ in $V_{i_1}$, disjoint from 
some $V_{i_2}$.  Iterating this process we can find disjoint sets $U_n \subset V_1$, each 
homeomorphic to $U$, and Hausdorff converging to $x$.  Define $f: V_1 \cup U_0 \to V_1$ 
to be the identity on the complement of $\bigcup_n U_n$ and send $U_i$ to $U_{i+1}$ by 
a homeomorphism as in Observation \ref{obs:shift}.  

If instead $y \less x$, then take any neighborhood $U$ of $y$ disjoint from $V_x$ and small enough so that $V_x$ contains a homeomorphic copy of $U$.  Since $y \less x$, this copy lies in $V_x - \{x\}$, and we may repeat the same line of argument above.  
\end{proof} 

%-------------------------------
\section{Classification of locally CB mapping class groups} 
\label{sec:loc_CB}

We now prove properties of locally CB mapping class groups, building towards our general classification theorem. Recall that we have the following notational convention.  

\begin{notation}
If  $K \subset \Sigma$ is a finite type subsurface, we denote by $\calV_K$ the identity neighborhood consisting of mapping classes of homeomorphisms that restrict to the identity on $K$.  
\end{notation}

\begin{lemma} \label{lem:nondispCB}
Let $K \subset \Sigma$ be a finite type subsurface.  If there exists a finite type, nondisplaceable (possibly disconnected) subsurface $S$ in $\Sigma-K$, then $\calV_K$ is not CB.  If this holds for every finite type $K \subset \Sigma$, then $\mcg(\Sigma)$ is not locally CB.
\end{lemma}

\begin{proof} 
Let $K$ be such a surface.  Enlarging $K$ and hence shrinking $\calV_K$ if needed, we may assume that all complimentary regions of $K$ have infinite type.  

Let $S$ be a nondisplaceable subsurface contained in $\Sigma - K$.  Now enlarging $S$ if needed, we may assume that it still remains in the complement of $K$, but is such that each component of $S$ has high enough complexity so that the length function $\ell_S$ defined in Section \ref{sec:nondisplace} will be unbounded.  
As in Proposition \ref{prop:disconnected}, this gives a length function which is unbounded on $\calV_K$, hence on $\calV$, so $\mcg(\Sigma)$ is not locally CB.  

Since each neighborhood of the identity in $\mcg(\Sigma)$ contains a set of the form $\calV_K$ for some finite type $K$, $\mcg(\Sigma)$ is locally CB if and only if some such set is CB. 
\end{proof}

Going forward, we reference the partial order $\less$ defined in Section \ref{sec:order}.

\begin{lemma} 
If $\mcg(\Sigma)$ is locally CB, then the number of distinct maximal types under $\less$ is finite. 
\end{lemma}

\begin{proof}
We prove the contrapositive.  Suppose that there are infinitely many distinct maximal types.  
Let $K$ be any subsurface of finite type.  By Lemma \ref{lem:nondispCB}, it suffices to find a nondisplaceable subsurface contained in $\Sigma-K$, which we do now.  

To every end $x \in E$ of maximal type, let $\sigma(x)$ denote the set of connected components of $\Sigma-K$ which contain ends from $E(x)$.  
Since $\Sigma-K$ has finitely many connected components, by the pigeonhole principle, there are two 
ends $x$ and $y$ with $x \nsim y$ but $\sigma(x) = \sigma(y)$.  
That is, each complimentary region of $\Sigma - K$ that has an end from $E(x)$ also contains ends from $E(y)$, and vice versa.  Fix any $z \in E$ with $z \nsim x$ and $z \nsim y$.  

Construct a surface $S$ as follows.  For each component $\tau$ of $\sigma(x)$, take a three-holed sphere subsurface of $\tau$ so that the complementary regions of the three-holed sphere separate $E(x)$ from $E(y)$ and $E(z)$ in $\tau$.  That is to say, one complimentary region contains only ends from $E(x)$ and none from $E(y)$ or $E(z)$, while another contains only ends from $E(y)$ and none from $E(x)$ or $E(z)$, and the third containing at least some points of $E(z)$ (possibly those from another complimentary region of $K$).    Let $S$ be the union of these three holed spheres.   Thus, each end from $E(x)$ is the end of some complimentary region of $S$ which has no ends of type $y$, and vice versa.   

We claim that $S$ is non-displaceable.   For if $S_i$ is a connected component of $S$, then one complimentary region of $S_i$ contains ends from $E(x)$, but none from $E(y)$.  By invariance of $E(x)$ and $E(y)$, if some homeomorphic image $f(S_i)$ were disjoint from $S$, then we would have to have $f(S_i)$ contained in one of the complimentary regions of $S$ containing points of $E(x)$.  However, this region contains no points of $E(y)$ or $E(z)$, contradicting our construction of $S_i$.  
Hence, $S$ is non-displaceable and,
by Lemma~\ref{lem:nondispCB}, $\mcg(\Sigma)$ is not locally CB. 
\end{proof}

\begin{proposition} \label{prop:partition} 
If $\mcg(\Sigma)$ is locally CB, then there is a partition 
\[ E = \bigsqcup_{A \in \calA} A \]
where $\calA$ is finite, each $A \in \calA$ is clopen, and $\calM(A) \subset \calM(E)$.  
Moreover, this decomposition can be realized by the complimentary regions to a finite type surface $L \subset \Sigma$ with $|\calA|$ boundary components, either of zero genus or of finite genus equal to the genus of $\Sigma$.  
\end{proposition} 

This will be a quick consequence of the following stronger result.

\begin{proposition} \label{prop:partition2} 
Suppose that $\mcg(\Sigma)$ is locally CB.  Then a CB neighborhood of the identity can be taken to be $\calV_K$ where $K$ is a finite type surface with the following properties

\begin{enumerate}
\item Each connected component of $\Sigma - K$ is of infinite type and has 0 or infinite genus.
\item The connected components of $\Sigma - K$ partition $E$ as 
$$E = \bigsqcup_{A \in \calA} A \sqcup \bigsqcup_{P \in \mathcal{P}} P $$
where each $A \in \calA$ is self-similar, and for each $P \in \mathcal{P}$, there exists $A \in \calA$ such that $P$ is homeomorphic to a clopen subset of $A$, and
\item For all $A \in \calA$, the maximal points $\calM(A)$ are maximal in $E$, and $\calM(E) = \sqcup_{A \in \calA} \calM(A)$
\end{enumerate}
\end{proposition}

\begin{proof}[Proof of Proposition \ref{prop:partition2}]
Suppose that $\calV$ is a CB neighborhood of identity in $\mcg(\Sigma)$.  Let $K$ be a finite type surface such that $\calV_K \subset \calV$, so $\calV$ is also CB.  Enlarging $K$ if needed (and hence shrinking $\calV_K$), we may assume that each complimentary region to $K$ has either zero or infinite genus.  
Since $\less$ has only finitely many maximal types, enlarging $K$ further, we may assume that its complimentary regions separate the different maximal types, and further if for some maximal $x$ the set $E(x)$ is finite, then all the ends from $E(x)$ are separated by $K$.  Thus, complimentary regions to $K$ have either no end from $\calM(E)$, a single end from $\calM(E)$ or a Cantor set of ends of a single type from $\calM(E)$.  

Our goal is to show that the complimentary regions containing ends from $\calM(E)$ are all self-similar sets, and the end sets of the remaining regions have the property desired of the sets $P \in \mathcal{P}$ described above.  It will be convenient to introduce some terminology for the set of ends of a complimentary region to $K$, so call such a subset of $E$ a {\em complimentary end set}.  

For simplicity, assume as a first case that for each maximal type $x$, the set $E(x)$ is {\em finite}.
Fix a maximal type point $x \in E$, and let $B_1, B_2, \ldots B_k \subset E$ be the 
complementary end sets whose maximal points lie in $E(x)$.  We start by showing that 
at least one of the sets $B_i$ is self-similar.   Let $x_i$ denote the maximal point in $B_i$.   Let $U_i$ be any clopen neighborhood of $x_i$ in $B_i$.   Since $x_i \in E(x)$, we may find smallr neighborhoods $V_i \subset U_i$ such that each $U_i$ contains a homeomorphic copy of $V_j$, for all $j = 1, 2, \ldots, k$.  
Let $S \subset \Sigma-K$ be a subsurface, homeomorphic to the disjoint union of $k$ pairs of pants, such that the complimentary regions of the $i$th pair of pants partitions the ends of $\Sigma$ into $V_i$, $B_i - V_i$ 
and $E - B_i$.  

Since $\calV_K$ is assumed CB, the surface $S$ is displaceable by Lemma \ref{lem:nondispCB}, so at least one of the connected components of $S$ can be moved disjoint from $S$ by a homeomorphism.  Since $E(x)$ is homeomorphism invariant, we conclude that there is a copy of $B_j$ in some $V_i$, possibly with $i \neq j$.  Our choice of $V_i$ now implies that there is in fact a homeomorphic copy of $B_j$ in $U_j$.   Thus, we have shown that, for any neighborhoods $U_i$ of $x_i$, there exists $j$ such that $U_j$ contains a copy of $B_j$.  Applying this conclusion to each of a nested sequence of neighborhoods of the $x_i$ which give a neighborhood basis, we conclude that some $j$ must satisfy this conclusion infinitely often, giving us some $B_j$ which is self-similar.  

Since $x_i$ are the unique maximal points of $B_i$, this implies that each $x_i$ has a neighborhood $M_i$ homeomorphic to $B_j$, i.e. a self-similar set.   Repeating this process for all of the distinct maximal types, we conclude that each maximal point has a self-similar neighborhood.  Fix a collection of such neighborhoods. Since this collection is finite we may enumerate them $A_1, A_2, \ldots A_n$.  

Now for each non-maximal point $y$, Lemma \ref{lem:consume} implies that there exists a neighborhood $P_y$ of $y$ so that $P_y \cup A_i$ is homeomorphic to some $A_i$, a neighborhood of a maximal point that is a successor of $y$.  
Since $E - \sqcup_i A_i$ is compact, finitely many such neighborhoods $P_y$ cover $E - \sqcup_i A_i$.  Enlarging $K$, we may assume that it partitions the end sets into the disjoint union of such sets of the form $P_y$ and $A_i$.  
This concludes the proof in the case where $\calM$ is finite.

Now we treat the general case where, for some maximal types, the set $E(x)$ is a Cantor set.  The strategy is essentially the same.  We use the following lemma, which parallels the argument just given above.  
\begin{lemma} \label{lem:SS_nbhd}
Keeping the hypotheses of the Proposition, let $x$ be a maximal type with $E(x)$ a Cantor set. Then $x$ has a neighborhood which is self-similar.  \end{lemma} 

\begin{proof}[Proof of Lemma]
Let $A_1, \ldots A_k$ be the complimentary end sets which contain points of $E(x)$, and fix a maximal end $x_i$ in each $A_i$.   As before, we start by showing that, for some $j$, every neighborhood of $x_j$ contains a homeomorphic copy of $A_j$, so in particular $A_j$ is self-similar.    Let $U_i$ be a neighborhood of $x_i$.  For each $z \in E(x)$, let $V_z$ be a neighborhood of $z$ such that each of the sets $U_i$ contains a homeomorphic copy of $V_z$.  Since $E(x)$ is compact, finitely many such $V_z$ cover $E(x)$, so from now on we consider only a finite subcollection that covers.   Let $S \subset \Sigma-K$ be a subsurface homeomorphic to the union of $k$ disjoint $n$-holed spheres, where $n$ is chosen large enough so that each complimentary region of $S$ has its set of ends either contained in one of the finitely many $V_z$, or containing all but one of the sets $A_i$.   

Again, since $E(x)$ is invariant, and $S$ is displaceable, this means that there is some $V_z$ and some $A_j$ such that $V_z$ contains a homeomorphic copy of $A_j$.  Thus, by definition of $V_z$, we have that $U_j$ contains a homeomorphic copy of $A_j$.  
Repeating this for a nested sequence of neighborhoods of the $x_i$, we conclude that some $x_j$ satisfies this infinitely often. 
This means that $A_j$ is a stable neighborhood of $x_j$, hence by Lemma \ref{lem:stable_nbhd}, each point of $E(x)$ has a stable neighborhood, which is necessarily a self-similar set.  
\end{proof} 

At this point one can finish the proof exactly as in the case where all $E(x)$ are finite, by fixing a finite cover of $\cup_{x \in \calM(E)} E(x)$ by stable neighborhoods, and using Lemma \ref{lem:consume} as before. 
\end{proof}

\begin{proof}[Proof of Proposition \ref{prop:partition}]
Let $E = \bigsqcup_{A \in \calA} A \sqcup \bigsqcup_{P \in \mathcal{P}} P $ be the decomposition given by Proposition \ref{prop:partition2}.  By construction of the sets $P$ and Lemma \ref{lem:consume}, for each $P \in \mathcal{P}$, there exists $A \in \calA$ such that $P \sqcup A \cong A$.  Applying this to each $P$ iteratively, we conclude that $E$ is homeomorphic to the disjoint union $\sqcup_{A \in \calA} A$.   We may now take $L$ to be a subset of $K$.  
\end{proof}

With this groundwork in place, we can prove our first major classification theorem.

\begin{theorem} \label{thm:classificationCB}
$\mcg(\Sigma)$ is locally CB if and only if %either $\Sigma$ is 0 or infinite genus and has self-similar end space (in which case it is globally CB), or 
there is a finite type surface $K$ such that the complimentary regions of $K$ each have infinite type and 0 or infinite genus, and partition of $E$ into finitely many clopen sets 
$$E = \left( \bigsqcup_{A \in \calA} A \right) \sqcup \left( \bigsqcup_{P \in \mathcal{P}} P \right)$$
with the property that: 

\begin{enumerate}
\item Each $A \in \calA$ is self-similar, $\calM(A) \subset \calM(E)$, and $\calM(E) = \sqcup_{A \in \calA} \calM(A)$.  
\item Each $P \in \mathcal{P}$ is homeomorphic to a clopen subset of some $A  \in \calA$.
\item For any $x_A \in \calM(A)$, and any neighborhood $V$ of the end $x_A$ in $\Sigma$, there is $f_V \in \Homeo(\Sigma)$ so that $f_V(V)$ contains the complimentary region to $K$ with end set $A$.
\end{enumerate}
Moreover, in this case $\calV_K := \{g \in \Homeo(\Sigma) : g|_K = id \}$ is a CB neighborhood of the identity, and 
$K$ may always be taken to have genus zero if $\Sigma$ has infinite genus, and genus equal to that of $\Sigma$ otherwise, and a number of punctures equal to the number of isolated planar (not accumulated by genus) ends of $\Sigma$. 
\end{theorem} 

Note that the case where $K = \emptyset$ implies that $\Sigma$ has zero or infinite genus and self-similar end space, in which case we already showed that $\calV_\emptyset = \mcg(\Sigma)$ is CB.  
The reader may find it helpful to refer to Figure \ref{fig:locCB} for some very basic examples, all with $\calP = \emptyset$, and keep this in mind during the proof. 

\begin{proof}[Proof of Theorem \ref{thm:classificationCB} $(\Rightarrow)$]
The forward direction is obtained by a minor improvement of Proposition \ref{prop:partition2}.   
Assume $\mcg(\Sigma)$ is CB.  Let $K \subset \Sigma$ be a finite type surface with $\calV_K$ a CB neighborhood of the identity and the properties given in Proposition \ref{prop:partition2}.  We may enlarge $K$ if needed so that each of its boundary curves are separating, and so that whenever some maximal type $x$ has $E(x)$ homeomorphic to a Cantor set, then $E(x)$ is contained in at least two complimentary regions to $K$.  This latter step can be done as follows: if $A$ is the unique complementary region of $K$ containing the Cantor set $E(x)$, then add a 1-handle to $K$ that separates $A$ into two clopen sets, each containing points of $E(x)$.   
Since $A$ is self-similar, each point of $\calM(A)$ has a stable neighborhood by \ref{lem:stable_easy}, and so the two clopen sets of our partition are again each self-similar and each homeomorphic to $A$.  

Thus, we assume $K$ now has these properties, and let $E = \left( \sqcup_{A \in \calA} A  \right) \sqcup \left( \sqcup_{P \in \mathcal{P}}P \right)$ be the resulting decomposition of $E$, with $\Sigma_A$ and $\Sigma_P$ denoting the connected component of $K$ with end space $A$ or $P$ respectively.  
We need only establish that the third condition holds. 
Fix $A$, let $x_A \in \calM(A)$, and let $V \subset \Sigma$ be a neighborhood of the end $x_A$, and let $E(V)\subset A$ denote the end space of $V$.  We may without loss of generality assume that $V$ has a single boundary component.    If $|E(x)| > 1$, 
then by construction there exists $B \neq A \in \calA$ with $E(x) \cap B \neq \emptyset$.  
Since points of $E(x)$ have stable neighborhoods, $B \cup (A - E(V))$ is homeomorphic to $B$. Moreover, if $\Sigma_A$ has infinite genus, then $V$ and $\Sigma - \Sigma_A$ and $\Sigma - V$ all do as well, while if $\Sigma_A$ has genus 0, then so does $V$, and both complementary regions are of the same genus as well (equal to the genus of $\Sigma$).  
Thus, by the classification of surfaces there exists some $f_V \in \mcg(\Sigma)$ taking $V$ to $\Sigma_A$, which is what we needed to show.

Now suppose instead $|E(x)| = 1$. Here we will use the displaceable subsurfaces condition to find $f_V$.  
Let $S$ be a pair of pants in the complement of $K$, with one boundary component equal to $\partial V$ and another homotopic to $\partial \Sigma_A$.  Since $S \subset (\Sigma - K)$, it is displaceable, so let $f$ be a homeomorphism displacing $S$. Since $E(x) = x_A$ is an invariant set, up to replacing $f$ with its inverse, we have $f(S) \subset V$.  
If, as a first case, there exists a maximal end $y \nsim x$, then $E(y)$ is also an invariant set.  Thus, $f(\Sigma_A) \subset V$, hence we may take $f_V = f^{-1}$ and have $f_V(V) \supset A$.  

If, as a second case, $\Sigma$ has finite genus, then $f(\Sigma - \Sigma_A)$ necessarily contains all the genus of $\Sigma$, hence again we have $f(\Sigma_A) \subset V$.   Finally, if neither of these two cases holds, then $\Sigma$ has infinite or zero genus, and a unique maximal end, so $|\calA| = 1$ and $E$ is self-similar.  Thus, $\mcg(\Sigma)$ is CB by Proposition \ref{prop:SSCB}.  
\end{proof}

\begin{proof}[Proof of Theorem \ref{thm:classificationCB} $(\Leftarrow)$]
For the converse, the case where $K = \emptyset$, $\Sigma$ has 0 or infinite genus and a self-similar end space is covered by Proposition \ref{prop:SSCB}.  
 
So suppose $\Sigma$ is not 0 or infinite genus with a self-similar end space, but instead we have a finite type surface $K$ with the properties listed.  We wish to show that $\calV_K$ is CB.  
Let $T \subset \Sigma$ be a finite type surface with $\calV_T \subset \calV_K$, i.e. $T \supset K$.  We need to find a finite set F and some $n$ such that $(F \calV_T)^n$ contains $\calV_K$.   

For each $A \in \calA$, fix $x_A \in \calM_A$ and let $V_A$ be the connected component of $T$ containing $x_A$. 
Let $f_V$ be the homeomorphism provided by our assumption.   Also, for each $P \in \mathcal{P}$, choose a homeomorphism $f_P$ of $\Sigma$ that exchanges $P$ with a clopen subset of some $A \in \calA$ which is homeomorphic to $P$.  Let $F$ be the set of all such $f_V^{\pm1}$ and $f_P^{\pm 1}$.  

Now suppose $g \in \calV_K$.  We can write $g$ as a product of $|\calA| + |\mathcal{P}|$ homeomorphisms, where each one is supported on a surface of the form $\Sigma_A$ for $A \in \calA$ or $\Sigma_P$ for $P \in \mathcal{P}$ (adopting our notation from the previous direction of the proof).  

If some such homeomorphism $g_A$ is supported on $\Sigma_A$, then $f_V^{-1} g_A f_V$ restricts to the identity on $T$, so $g_A \in F\calV_T F$.  For a homeomorphism $g_P$ supported on $\Sigma_P$, then $f_P^{-1} g_P f_P$ is supported in $\Sigma_A$, so $g_P \in F^2\calV_T F^2$.  This shows that $g \in (F^2 \calV_T F^2)^{|\calA| + |\mathcal{P}|}$, which is what we needed to show.  
\end{proof} 

\subsection{Examples}
While the statement of Theorem \ref{thm:classificationCB} is somewhat involved, it is practical to apply in specific situations. Below are a few examples illustrating some of the subtlety of the phenomena at play.  
The first is an immediate consequence. 

\begin{corollary}
If $\Sigma$ has finite genus and countable, self-similar end space, then $\Sigma$ is not locally CB
\end{corollary}

As a more involved example, we have the following. 

\begin{corollary} \label{cor:infinite-rank-sort-of}
Suppose that $\Sigma$ has finite genus and self-similar end space, with a single maximal end $x$, but infinitely many distinct  immediate precursors to $x$.  Then $\mcg(\Sigma)$ is not locally CB.  
\end{corollary}

As a concrete example, one could construct $E$ by taking countably many copies of a Cantor set indexed by $\N$, all sharing a single point in common and Hausdorff converging to that point, with the $n$th copy accumulated everywhere by points locally homeomorphic to $\omega^n +1$.  

\begin{proof} 
If $\mcg(\Sigma)$ were locally CB, then we would have a finite type surface $K$ as in Theorem \ref{thm:classificationCB}. 
Since $\calM$ is a singleton, $\calA = \{A\}$, and $x_A = x$, and $\Sigma_A$ is some neighborhood of the end $x$.  However, by construction, $E - A$ contains ends of only finitely many types.   Thus, we may choose a smaller neighborhood $V$ of $x$ so that $\Sigma - V$ has more distinct types of ends.  Then $\Sigma - V$ cannot possibly be homeomorphic to $\Sigma - \Sigma_A$, so no such $f_V$ exists.  
\end{proof}

By contrast, if $\Sigma$ is finite genus with end space equal to a Cantor set, or attained by the construction in Corollary \ref{cor:infinite-rank-sort-of} but replacing $\N$ with a finite number, then $\mcg(\Sigma)$ is locally CB.    We draw attention to a specific case of this to highlight the role played by $\calA$ and $\mathcal{P}$.

\begin{example} \label{ex:weird_ex}
Let $\Sigma$ be a surface of finite, nonzero genus $g$, with $E$ homeomorphic to the union of a Cantor set $C$ and a Cantor set $D$, and a countable set $Q$, with $C \cap D = \{x\}$ and the accumulation points of $Q$ equal to $D$, as illustrated in Figure \ref{fig:PandA}
Then by Theorem \ref{thm:classificationCB}, a CB neighborhood of the identity in $\mcg(\Sigma)$ can be taken to be $\calV_K$ where $K$ is a finite type subsurface of genus $g$ with two boundary components, with one complimentary region to $K$ having $x$ as an end, and the other containing points of both $C$ and $D$.   In this case, $\calA$ and $\calP$ are both singletons, with one complimentary region in each. 

The set $E$ itself is self-similar, and the decomposition into self similar sets given by Proposition \ref{prop:partition} is trivial.  However, if $K'$ is a finite type subsurface realizing this decomposition (with a single complimentary region), then $\calV_{K'}$ is not a CB set.  Indeed, one may find a nondisplaceable subsurface in the complement homeomorphic to a three holed sphere, where one complimentary region has $x$ as and end, one contains all the genus of $\Sigma$ but no ends, and the third contains points of $C$, for example.

  \begin{figure*}[h]
     \centerline{ \mbox{
 \includegraphics[width = 1.7in]{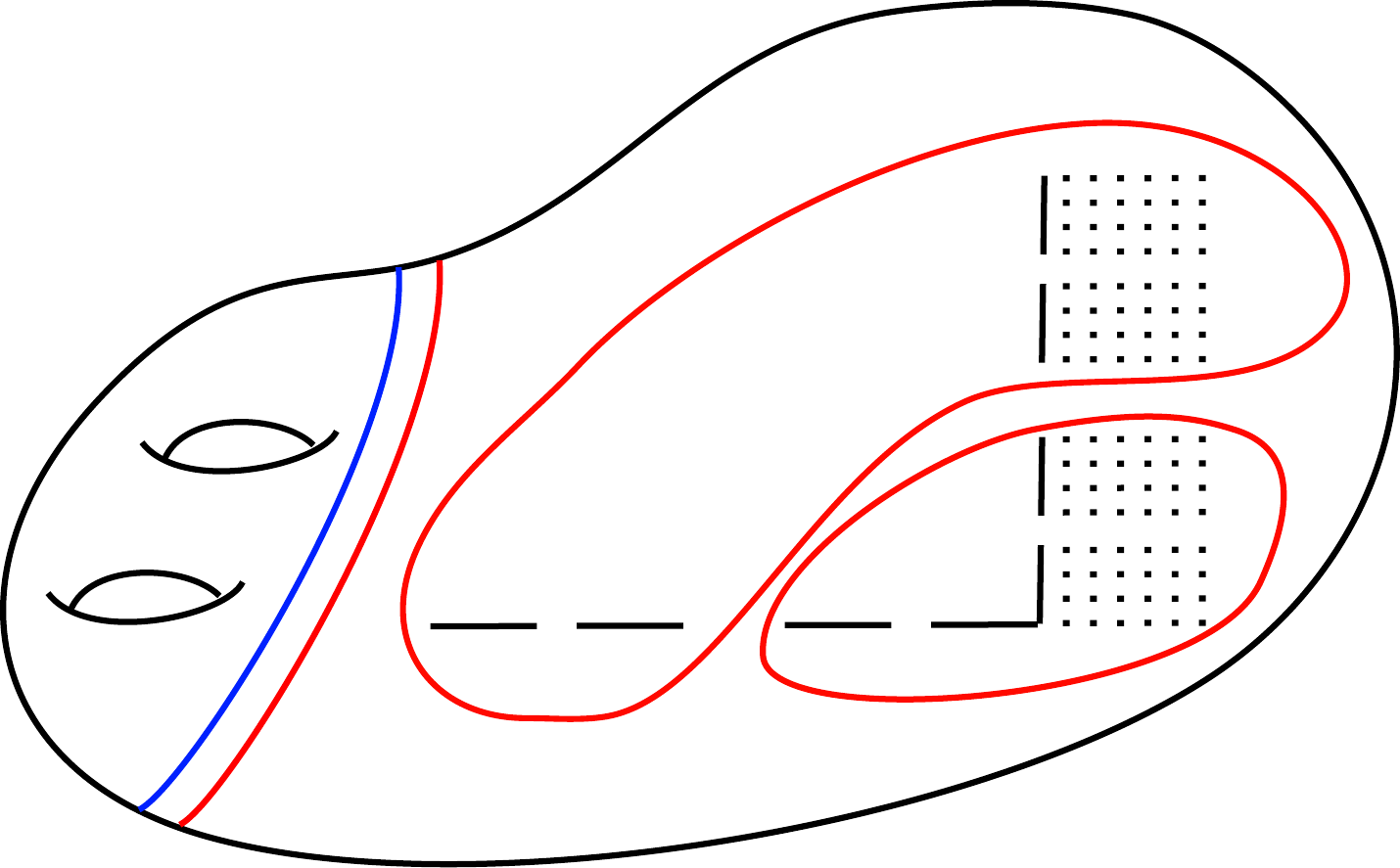}}}
 \caption{The subsurface with red boundary defines a CB neighborhood, while the smaller subsurface with blue boundary does not.}
  \label{fig:PandA}
  \end{figure*}

\end{example} 

%------------------------------------------
\section{CB generated mapping class groups}  \label{sec:CBgen}
In this section we give general criteria for when mapping class groups are CB generated, building towards the proof of Theorem \ref{thm:CB_generated}.

\subsection{Two criteria for CB generation} \label{sec:criteria}

\begin{notation}
For a subset $X \subset E$, we say a family of neighborhoods $U_n$ in $E$ 
{\em descends to $X$} if $U_n$ are nested, meaning $U_{n+1} \subseteq U_n$, and if 
$\bigcap_{n\in \N} U_n = X$. As a shorthand, we write $U_n \searrow X$. If $X= \{x\}$ is a singleton, we abuse notation slightly and write 
$U_n \searrow x$ and say $U_n$ descends to $x$.   
\end{notation} 

\begin{definition}[Limit type] We say that an end set $E$ is {\em limit type} if there is a finite index subgroup G of 
$\mcg(\Sigma)$, a $G$--invariant set $X\subset E$, points $z_n \in E$, indexed by $n \in \N$ which are pairwise inequivalent, 
and a family of neighborhoods $U_n \searrow X$ such that 
\[
E(z_n) \cap U_n \not = \emptyset, \qquad 
E(z_n) \cap U_0^c \not = \emptyset,
\qquad\text{and}\qquad
E(z_n) \subset \big(U_n \cup U_0^c\big). 
\]
Here $U_0^c = E - U_0$ denotes the complement of $U_0$ in $E$. 
\end{definition} 

The following example explains our choice of the terminology ``limit type''.

\begin{example} \label{ex:limit_type}
Suppose that $\alpha$ is a countable limit ordinal, and $E \cong \omega^\alpha \cdot n + 1$, with $n \geq 2$, and $E^G = \emptyset$.   To see that this end space is limit type, take $G$ to be the finite index subgroup pointwise fixing the $n$ maximal ends.   Fix a maximal end $x$ and a clopen neighborhood $U_0$ of $x$ disjoint from the other maximal ends, and let $U_n \subset U_0$ be nested clopen sets forming a neighborhood basis of $x$.   Since $U_n - U_{n+1}$ is closed, there is a maximal ordinal $\beta_n$ such that $U_n$ contains points locally homeomorphic to $\omega^{\beta_n} +1$.  Passing to a subsequence we may assume that all of these are distinct, and one may choose $z_n \in U_n$ to be a point locally homeomorphic to $\omega^{\beta_n} +1$.    Note that necessarily the sequence $\beta_n$ converges to the limit ordinal $\alpha$.   The assumption that $n \geq 2$ ensures that the sets $E(z_n)$ contain points outside of $U_0$, and we require this in the definition to ensure that $E$ is not self-similar.  
\end{example}

\begin{lemma}[Limit type criterion] \label{lem:limit-type} 
If an end set $E$ has limit type, then $\mcg(\Sigma)$ is not CB generated. 
\end{lemma}

\begin{proof}
Let $G$, $X$, $U_n$ and $z_n$ be as in the definition of limit type. We will show $G$ 
is not CB generated.  Since $G$ is finite index, this is enough to show that $\mcg(\Sigma)$ is not CB generated. 
Furthermore, since $\mcg(\Sigma)$ (and hence $G$) is assumed to be locally CB, it suffices to show that there is some neighborhood $\calV_G$ of the identity in $G$ such that for any finite set $F$, the set $F\calV_G$ does not generate $G$.   

Let $\calV_G$ be a neighborhood of the identity in 
$G$, chosen small enough so that, for every $g \in \calV_G$, and all $n>0$ we have 
$g(U_n) \subset U_0$ and $g(U_0^c) \cap U_n = \emptyset$. 

Let $F$ be any finite subset of $G$. Since $G$ preserves both the set $X$ and the set $E(z_n) \subset U_n \sqcup U_0^c$, there exists $N \in \N$ such that for all $n > N$, and all $f \in F$, we have 
\[
f \big( E(z_n) \cap U_n) \big) \subset U_n.
\]      
The same holds for elements of $\calV_G$.  

Fix such $n>N$, and let $x_n \in E(z_n) \cap U_n$ and $y_n \in E(z_n) \cap U_0^c$ .   Since $x_n \sim y_n$, there is a homeomorphism $h$ with $h(x_n)$ lying in a small neighborhood of $y_n$ contained in $U_0^c$.  
By our observation above, $h$ is not in the subgroup generated by $F\calV_G$, which shows that $G$ is not CB generated, as desired. 
\end{proof}

A second obstruction to CB-generation is the following ``rank" condition. 

\begin{definition}[Infinite rank] 
We say $\mcg(\Sigma)$ has \emph{infinite rank} if there is a finite
index subgroup $G$ of $\mcg(\Sigma)$, a closed $G$--invariant set $X$, neighborhood
$U$ of $X$ and points $z_n$, $n \in \N$, each with a {\em stable neighborhood} (see def. \ref{def:stable_nbhd}) such that 
\begin{itemize}
\item $z_n \notin E(z_m)$ if $m \neq n$, 
\item for all $n$, $E(z_n)$ is countably infinite and has at least one accumulation point in both $X$ and in $E-U$, and
\item the set of accumulation points of $E(z_n)$ in $U$ is a subset of $X$. 
\end{itemize}
If the above does not hold, we say instead that $\mcg(\Sigma)$ has {\em finite rank}.  
\end{definition} 

\begin{example}  \label{ex:infinite_rank}
A simple example of such a set is as follows.  Let $C_n$ be the union of a countable set and a Cantor set, with Cantor-Bendixson rank $n$, and $n$th derived set equal to the Cantor set.  For each $C_n$, select a single point $z_n$ of the Cantor set to be an end accumulated by genus.  Now create an end space $E$ by taking $\N$-many copies of each $C_n$, arranged so that they have exactly two accumulation points $x$ and $y$ (and these accumulation points are independent of $n$).   Then $X = \{x\}$  and the points $z_n$ satisfy the definition.  
\end{example} 

%\textcolor{blue}{PICTURE HERE} 

Examples of surfaces with {\em countable} end spaces and infinite rank mapping class groups are much more involved to describe.  (Note that these necessarily must have infinite genus.). It would be nice to see a general procedure for producing families of examples.  

\begin{lemma}[Infinite Rank Criterion]  \label{lem:infinite-rank} 
If $\mcg(\Sigma)$ has infinite rank, then it is not CB generated. 
\end{lemma}

\begin{proof}
Let $G$, $X$, $U$ and $z_n$ be as the definition of infinite rank. For every 
$z_n$, we define a function $\ell_n \from G \to \Z$ as follows.  
For $\phi \in G$, define
\[
\ell_n(\phi) = 
\Big| \big( E(z_n) \cap U \big) - \phi^{-1}(U) \Big| - 
\Big| \big( E(z_n) \cap \phi^{-1}(U) \big) - U \Big|
\]
That is, $\ell_n(\phi)$ is the the difference between the number of points in $E(z_n)$
that $\phi$ maps out of $U$ and the number of points in $E(z_n)$ that $\phi$ maps 
into $U$. 

Since $X$ is $G$--invariant and contains all of the accumulation points of $E(z_n)$ in $U$, 
the value of $\ell_n$ is always finite.  It is also easily verified that $\ell_n$ is a homomorphism.  Moreover, since each $z_n$ has a stable neighborhood (all of which are pairwise homeomorphic), for any finite collection $n_1, \ldots n_k$ one may construct, for each $i$, a ``shift" homomorphism $\phi_i$ supported on a union of disjoint stable neighborhoods of $E(z_{n_i})$, taking one stable neighborhood to the next, such that $\ell_{n_i}(\phi_i) = 1$ and $\ell_{n_j}(\phi_j) = 0$ for $j \neq i$.  
Finally, $\ell_n$ is continuous, in fact any neighborhood $\calV$ of the identity in $G$ which is small enough so that elements of $\calV$ fix the isotopy class of a curve separating $U$ from $E - U$, we will have $\ell_n(\calV) = 0$. 

Thus, 
we have for each $k \in \N$ a surjective, continuous homomorphism
\[
(\ell_{n_1}, \dots, \ell_{n_k}) \from G \to \Z^k
\]
which restricts to the trivial homomorphism on the neighborhood $\calV$ of the identity described above.   

By Theorem \ref{thm:CB_criterion}, any CB set is contained in a set of the form $(F \calV)^k$ for some finite set $F$ and $k \in \N$.  Given any such $F$, choose $j > |F|$.  Then restriction of $(\ell_{n_1}, \dots, \ell_{n_j})$ to the subgroup generated by $(F \calV)^k$ cannot be surjective, as $\calV$ lies in its kernel.  It follows that no CB set can generate $G$. 
Since $G$ is finite index in $\mcg(\Sigma)$, the same is also true $\mcg(\Sigma)$. 
\end{proof}

%-------------------------------------------------------
\subsection{End spaces of locally CB mapping class groups}   \label{sec:limit-type}
For the remainder of this section, we assume that $\mcg(\Sigma)$ is locally CB, our ultimate goal being to understand which such groups are additionally CB generated.  
Recall that Proposition \ref{prop:partition} gave a decomposition of $E$ into a disjoint union of self-similar sets homeomorphic to $A \in \calA$, realized by a finite type subsurface $L \subset K$.  However, as shown in example \ref{ex:weird_ex}, the neighborhood $\calV_L$ might not be CB.  We now show that $\calV_L$ is CB generated.

\begin{lemma} \label{lem:CB_gen_nbhd}
Assume that $\mcg(\Sigma)$ is locally CB. 
Let $L$ be a finite type surface whose complimentary regions realize the decomposition $E = \bigsqcup{A \in \calA}$ given by Proposition \ref{prop:partition}.   Then $\calV_L$ is CB generated.  

Furthermore, we may take $L$ to have genus zero if $\Sigma$ has infinite genus and genus equal to that of $\Sigma$ otherwise; and a number of punctures equal to the number of isolated planar (not accumulated by genus) ends of $\Sigma$ if that number is finite, and zero otherwise.  
\end{lemma}

For the proof, we need the following observation, which follows from well known results on standard generators for mapping class groups of finite type surfaces. 

\begin{observation} \label{obs:pmap}
Let $\Sigma$ be an infinite type surface, possibly with finitely many boundary components, and $S \subset \Sigma$ a finite type subsurface.  Then there is a finite set of 
Dehn twists $D$ such that for any finite type surface $S'$,
$\mcg(S')$ is generated by $D$ and $\calV_S$.
\end{observation}

In fact, akin to Lickorish's Dehn twist generators for the mapping class group of a surfaces of finite type, 
one can find a set $\calD$ of simple closed curves in $\Sigma$ such that every curve in $\calD$
intersects only finitely many other curves in $\calD$, and such that the set of Dehn twists 
around curves in $\calD$ generates the subgroup of $\mcg(\Sigma)$ consisting of mapping classes supported on finite type 
subsurfaces of $\Sigma$. (See \cite{PatelVlamis}). One can then take the set $D$ of Observation \ref{obs:pmap} to
be the set of Dehn twists around the curves in $\calD$ that intersect $S$. 

\begin{proof}[Proof of Lemma \ref{lem:CB_gen_nbhd}]
Let  $K$ be the surface given by Theorem \ref{thm:classificationCB}.  
For each $P \in \calP$, there exists $A \in \calA$ such that $A - \{x_A\}$ contains a 
homeomorphic copy of $P$.  Choose one such $A$ for each $P \in \calP$, and for $A \in \calA$ let $P_A$ denote the union of the elements of $\calP$ assigned to $A$.   
Let $L \subset K$ be a connected, finite type surface with $|\calA|$ boundary components, and such that the complimentary regions of $L$ partition $E$ into the sets $A \cup P_A$, as $A$ ranges over $\calA$.  We take $L$ to have the same number of punctures and genus as $K$.  For each $A$, let $\Sigma_A$ denote the complimentary region to $L$ with end space $A \cup P_A$.  

If $f \in \calV_{L}$, then $f$ can be written as a product of $|\calA|$ homeomorphisms, one supported on each surface $\Sigma_A$ (and hence identifiable with an element of 
$\mcg(\Sigma_A)$).  So it suffices to show, for each $A \in \calA$, that $\mcg(\Sigma_A)$ is generated by $\calV_K \cap \mcg(\Sigma_A)$, which is a CB subset of $\mcg(\Sigma)$, together with a finite set.  

Fix $A$, let $K'$ denote $K \cap \Sigma_A$, let $\Sigma_1$, $\Sigma_2$ \ldots 
$\Sigma_n$ denote the connected components of $\Sigma_A - K'$ with end spaces elements of $\calP$, and let $\Sigma'$ be the connected component with end space $A$.  Thus, $\calV_K \cap \mcg(\Sigma_A) = \mcg(\Sigma')$.   Let $T \subset \Sigma_A$ be a connected, infinite type subsurface consisting of the union of $\Sigma_1 \sqcup  \ldots  \sqcup \Sigma_n$ and an $(n+1)$-holed sphere in $\Sigma_A$, so that $T$ has a single boundary component.  

Recall from Proposition \ref{prop:partition2} that $P_A$ contains no maximal points, that $A \cup P_A$ is a self-similar set (and homeomorphic to $A$), and in particular we can find a copy of $P_A$ in any neighborhood of $x_A$.  
By the classification of surfaces, this implies that there is a homeomorphic copy of $T$ contained in $\Sigma'$, i.e. some homeomorphism $g_A$ of $\Sigma_A$ with $g_A(T) \subset \Sigma'$.  This homeomorphism and its inverse will form part of the desired finite set.  

Let $f \in \calV_{L}$.   
Since $\calM(A)$ is an invariant set, we may find a neighborhood $U$ of $\calM(A)$ in $\Sigma_A$, which we may take to be a (infinite type) subsurface of $\Sigma'$ with a single boundary component,  such that $f(U) \subset \Sigma'$.  
Let $T'$ be a homeomorphic copy of $T$ contained in $U$.  Thus, $f(T') \subset \Sigma'$, and so there exists $h \in \mcg(\Sigma')$ with $hf(T') = g_A(T)$, so $g_A^{-1}hf(T') = T$.
Since $g_A^{-1}hf$ is a homeomorphism taking $T'$ to $T$, we have 
\[
g_A^{-1}hf(\Sigma_A - T') = \Sigma_A - T \subset (\Sigma' \cup K').
\]  
Thus, there exists $h' \in \calV_L$ interchanging $g_A^{-1}hf(T)$ with $T'$, and such that $h'\circ (g_A^{-1}hf)$ agrees with $g_A$ on $T$.  It follows that 
\[g_A^{-1} \circ h'\circ g_A^{-1}hf\] 
is the identity on $T$ and it can be identified with a mapping class of $K' \cup \Sigma'$.  
Since all ends of $K' \cup \Sigma'$ are ends of the connected surface $\Sigma'$, after modifying by an element of $\mcg(\Sigma')$, we can take this to be a pure mapping class.  By Observation \ref{obs:pmap}, $\pmcg(K' \cup \Sigma')$ is generated by $\mcg(\Sigma')$ and a fixed finite set of Dehn twists, hence taking these Dehn twists together with 
$g_A^{\pm 1}$ gives the required finite set.  
\end{proof} 

Going forward, we will ignore the surface $K$ produced earlier that defined the CB neighborhood $\calV_K$, and instead use the surface $L$ which gives a simpler decomposition of the end space.   The sets $P \in \calP$ play no further role, and we focus on the decomposition 
$E = \bigsqcup_{A \in \calA} A$ given by the end spaces of complimentary regions to the surface $L$.   

\subsection*{Further decompositions of end sets}
Now we begin the technical work of the classification of CB generated mapping class groups.  As motivation for our next lemmas, consider the surface depicted in Figure \ref{fig:locCB} (left).  This surface has a mapping class group which is both locally CB and CB generated (we have not proved CB generation yet, but the reader may find it an illustrative exercise to attempt this case by hand).  
Here, the decomposition of $E$ given by the surface $L$ is $E = A \sqcup B \sqcup C$ where $A$ and $C$ are accumulated by genus, $A$ and $B$ are homeomorphic to $\omega +1$, and $C$ is a singleton.  
As well as a neighborhood of the identity of the form $\calV_L$, any generating set must include a ``handle shift" moving genus from $A$ into $C$ (see Definition \ref{def:handle_shift} below), as well as a ``puncture shift" that moves isolated punctures out of $A$ and into $B$.   If each handle was replaced by, say, a puncture accumulated by genus, one would need a shift moving these end types in and out of neighborhoods of $A$ and $C$ instead.    

To generalize this observation to other surfaces with more complicated topology, we need to identify types of ends of $\Sigma$ that accumulate at the maximal ends of the various sets in the decomposition.  The sets $W_{A,B}$ defined in Lemma \ref{lem:W_sets} and refined in Lemma \ref{lem:best_set} below pick out blocks of ends that can be shifted between elements $A$ and $B$ in $\calA$.  Ultimately, we will have to further subdivide these blocks to distinguish different ends that can be independently shifted, this is carried out in Section \ref{sec:tameCBgen}. 

\begin{lemma} \label{lem:W_sets}
Assume $E$ does not have limit type. Then 
\begin{itemize}
\item For every $A \in \calA$, there is a neighborhood $N(x_A) \subset A$ containing 
$x_A$ such that $A-N(x_A)$ contains a representative of every type in $A - \{x_A\}$.
\item For every pair $A, B \in \calA$, there is a clopen set $W_{A, B} \subset (A-N(x_A))$ 
with the property that
\[
E(z) \cap W_{A,B} \neq \emptyset 
\quad\Longleftrightarrow\quad 
E(z) \cap \big(A-\{x_A\}\big) \neq \emptyset \quad\text{and}\quad 
E(z) \cap \big(B-\{x_B\}\big) \neq \emptyset.
\]
\item There is a clopen set $W_A \subset (A-N(x_A))$ with the property that if
$E(z) \cap \big(A-\{x_A\}\big) \neq \emptyset$ and, for all $B \neq A$, 
$E(z) \cap \big(B-\{x_B\}\big) = \emptyset$ then $E(z) \cap W_A \neq \emptyset.$
\end{itemize}
\end{lemma}
In other words, $W_{A,B}$ contains representatives of every type of end that appears 
in both $A - \{x_A\}$ and $B - \{x_B\}$, and $W_A$ contains representatives of every type that appears only in $A$.  

We declare $W_{A,B} = \emptyset$ if $A - \{x_A\}$ 
and $B - \{x_B\}$ have no common types of ends, and similarly take $W_A = \emptyset$ if each types of end in $A$ appears also in some $B \neq A$.  

\begin{proof}
We start with the first assertion. If $\calM(A)$ is a Cantor set then we can take 
$N(x_A)$ to be any neighborhood of $x_A$ that does not contain all of $\calM(A)$, and 
the first assertion follows since $\calM(A)$ is the set of maximal points.  
Otherwise, $\calM(A) = \{x_A\}$.  Let $G$ be the finite index subgroup of $\mcg(\Sigma)$ that fixes 
$E(x_A)$ (which we know is finite). If such a neighborhood $N(x_A)$ does not exist, 
then there is a nested family of neighborhoods $U_n$ descending to $x_A$ 
and points $z_n \in U_n$ where $(E(z_n) \cap A) \subset U_n$.
Then letting $X = \{x_A\}$ and assuming $U_0 = A$, we see that $E$ has limit type. 
The contradiction proves the first assertion.

For the second assertion, fix $A$ and $B \in \calA$ and let 
$$X = \Big\{x \in E \mid  E(x) \cap A \neq \emptyset 
\quad\text{ and }\quad 
E(x) \cap B = \emptyset\Big\}.$$
Then $X \cap A$ is closed -- this follows since $A$ is closed, and if $x_n$ is a sequence of points in $X \cap A$ converging to $x_\infty$ but there is some point $z \in E(x_\infty) \cap B$, then any neighborhood of $z$ would contain homneomorphic copies of neighborhoods of $x_n$, for sufficiently large $n$, contradicting the fact that $E(x_n) \cap B = \emptyset$.  

Now consider a family of neighborhoods $U_n$ of $X \cap A$ with 
$U_n \searrow X$ and $U_0 \cap B = \emptyset$.  
Let $W_n = A - \big( U_n \cup N(x_A) \big)$.
Note that, since we have removed the neighborhood $U_n$ of $X$, every point in $W_n$ has 
a representative in $B$. We claim that, for some $N \in \N$, $W_N$ contains a 
representative of all points that appear in both $A$ and $B$, that is to say, $W_{A,B}$ 
can be taken to be $W_N$. 
To prove the claim, suppose for contradiction that it fails. Then after passing to a subsequence, we may find points 
$z_n$, all of distinct types, such that $z_n \in U_n$, $E(z_n) \cap A \neq \emptyset$, and $E(z_n) \cap B \neq \emptyset$. Since $E(z_n)$ intersects 
$U_0^c \supset B$, this implies that $E$ has limit type. The contradiction proves
the second assertion.  

For the third assertion, consider the closed set  
\[
X = \Big\{x \in E \mid  E(x) \cap A \neq \emptyset 
\quad\text{ and }\quad  E(x) \cap B = \emptyset, \quad \forall B \neq A \Big\}.
\]
Let $U$ be any clopen neighborhood of $X \cap A$ in $A$, and let $W_A = U - N(x_A)$.  Then by definition of $N(x_A)$, $(X \cap A) - N(x_A)$ contains a representative of every type appearing only in $A$, so this remains true of its clopen neighborhood $W_A$.  
\end{proof}

\begin{definition}
When $E(z)$ is countable, we will say that $z$ is a point {\em of countable type}. 

For $A, B \in \calA$, we define 
\[E(A,B) = \Big\{ z \in A - \{x_A\} : E(z) \cap B \neq \emptyset \Big\} \]
and define $E_{\rm mc}(A,B)$ (the ``maximal countable set") to be the subset of $E(A,B)$ consisting of points $z$ 
where $z$ is maximal in $E(A,B)$ and of countable type.  
\end{definition} 
 
Note that $E(A,B) \subset A$.  In particular $E(A,B) \neq E(B, A)$, and the same is true for the maximal countable sets $E_{\rm mc}(A,B)$.  
The following elementary observation will be useful.  

\begin{observation} \label{obs:Emc}
If $z$ is any point of countable type, then any accumulation point $p$ of $E(z)$ satisfies $z \less p$.   Thus, if $z \in \Emc(A,B)$, the only accumulation point of $E(z)$ in $A$ is $x_A$ and so $(E(z) \cap A) - N(x_A)$ is a finite set.  
\end{observation}

\begin{lemma} \label{lem:mc_finite}
Suppose $\mcg(\Sigma)$ has neither limit type nor infinite rank.  Then, for any $A,B \in \calA$, the set
$E_{\rm mc}(A,B)$ contains only finitely many different types. 
\end{lemma}

\begin{proof}
As a first case, suppose that $\calM(A)$ is a single point.  Let $G$ be the finite index subgroup of $\mcg(\Sigma)$ 
that fixes $x_A$ (recall that $E(x_A)$ is finite). Now $X = \{ x_A\}$ is $G$--invariant and since $\mcg(\Sigma)$
does not have infinite rank, we can take $U=A$ and conclude that 
$E_{\rm mc}(A,B)$ has finitely many different types. 

Otherwise, $\calM(A)$ is a Cantor set.  If $E(x_A)$ does not intersect $B$, we can 
take $X = E(x_A)$ and $U =B^c$. Then $X$ is $\mcg(\Sigma)$--invariant and again
the fact that $\mcg(\Sigma)$ does not have infinite rank implies that 
$E_{\rm mc}(A,B)$ has finitely many different types. 

If $\calM(A)$ is a Cantor set and $E(x_A)$ intersects $B$, then $E(x_A)$ intersects
$E(A,B)$ and hence the only maximal point in $E(A,B)$ is $x_A$ itself, thus $\Emc(A,B)$ is empty.  
\end{proof} 

\subsection{Tame end spaces} 

\begin{definition} \label{def:tame}
An end space $E$ is {\em tame} if any $x \in E$ that is either maximal or an immediate predecessor of a maximal point has a {\em stable neighborhood} as in Definition \ref{def:stable_nbhd}.
\end{definition} 

If $\Sigma$ has locally CB mapping class group, then Theorem~\ref{thm:loc_CB} 
implies that maximal points have stable neighborhoods, so half of the tameness condition is satisfied.  The other half is an assumption that will be used in the next two sections.  
While this seems like a restrictive hypothesis, the class of tame surfaces is very large.  In fact, the following problem seems to be challenging, as the examples of non-tame surfaces (excluding those which are self-similar, see Example \ref{ex:non_tame} below) which we can easily construct all seem to have infinite-rank or limit-type like behavior.  

\begin{problem} \label{prob:non_tame}
Does there exist an example of a non-tame surface whose mapping class group has nontrivial, well-defined quasi-isometry type (i.e. is locally, but not globally, CB
and CB generated)?  
\end{problem} 

\begin{example}[Non-tame surfaces] \label{ex:non_tame} 
Suppose $\{z_n\}_{n \in \N}$ is a sequence of points in $E$ which are non-comparable, i.e. for all $i \neq j$ we have neither $z_i \lesseq z_j$ nor $z_j \lesseq z_i$.  An end space containing such a sequence may be constructed, for instance, as in Example \ref{ex:infinite_rank}, and even (as in that example) have the property that each $z_i$ admits a stable neighborhood $V_n$.  Let $D$ denote a set consisting of the disjoint union of one copy of each $V_n$ and a singleton $x$, so that the sets $V_n$ Hausdorff converge to $x$.   Then $x$ is a maximal point, but fails the stable neighborhood condition in the definition of tame, since the homeomorphism type of small neighborhoods of $x$ do not eventually stabilize.  

A surface with end space $D$ fails the condition of Theorem \ref{thm:loc_CB} so is not locally CB, but one can easily modify this construction to provide locally, and even globally, CB examples.  For instance, let $E$ be the disjoint union of countably many copies of $D$, arranged to have exactly $k$ accumulation points.  If $k = 1$, the end space constructed is self-similar, with the sole accumulation point the unique maximal point.  If $k > 1$, the end space may be partitioned into finitely many self-similar sets satisfying the condition of Theorem \ref{thm:loc_CB}, but has immediate predecessors to the maximal points with no tame neighborhood.   (However, we note that this example is infinite rank, so the mapping class group of a surface with this end type is not CB generated.) 
\end{example}

The main application of the tameness condition is that it allows us give a standard form to other subsets of $E$.   As a first example of this, we have the following.  

\begin{lemma} \label{lem:best_set}
Suppose that $\Sigma$ has tame end space.  
Then, under the hypotheses of Lemma \ref{lem:mc_finite}, the sets $W_{A,B}$ from 
Lemma \ref{lem:W_sets} can be chosen so that for any $z \in E_{\rm mc}(A,B)$, the set 
$E(z) \cap W_{A,B}$ is a singleton.   Such a choice specifies a set unique up to 
homeomorphism, and in this case $W_{A,B}$ is homeomorphic to $W_{B,A}$.  
\end{lemma}

\begin{proof} Fix a choice of set $W_{A,B}$ as given by Lemma \ref{lem:W_sets}.  
For each $z \in E_{\rm mc}(A,B)$,  choose disjoint stable neighborhoods
around every point in the finite set $E(z) \cap W_{A,B}$ (this set is finite by Observation \ref{obs:Emc}) and remove all but one neighborhood, leaving the rest of $W_{A,B}$ unchanged . Denote this new set by $W'(A,B)$.  Since one such neighborhood remains, any type that was represented in $W_{A,B}$ is 
still represented there, so it satisfies the hypotheses of Lemma \ref{lem:W_sets}.  
If $E$ is countable, any such choice will produce a homeomorphic set by the definition of stable neighborhood.  
It remains to show that this holds generally (i.e. even when some maximal points are contained in Cantor sets this construction still produces a set unique up to homeomorphism), and that 
$W'_{A,B}$ is homeomorphic to $W'_{B,A}$.  We prove both assertions simultaneously. 

Let $z_1, \dots, z_k \in W'_{A,B}$ be points of maximal countable type and let
$V_1, \dots, V_k$ be the chosen disjoint stable neighborhoods of these points in $W'_{A,B}$. 
Let $W = W'_{A,B} - \cup_i V_i$. Similarly, choose $V_1', \dots, V_k'$ to be 
stable neighborhood of points of maximal countable type in $W_{B,A}$ such that
$V_i$ is homeomorphic to $V_i'$ and let $W' = W'_{B,A} - \cup_i V_i'$.  
We start by showing that 
\[
W \cup W_{B,A} \cong W_{B,A}. 
\]
This is because, for any point in $x \in W$, there is a point $y \in W'_{B,A}$ of maximal type where
$y$ is an accumulation point of $E(x)$. Hence, by Lemma \ref{lem:consume}, there is a neighborhood $U_x$ of $x$ and stable neighborhood $V_y$ of $y$ such that $U_x \cup V_y$ is homeomorphic to $V_y$. 
Since $W$ is compact, 
finitely many such neighborhoods are enough to cover $W$ and, shrinking these neighborhoods if needed,
we can write $W$ as the  disjoint union of finitely many such neighborhoods. Thus, $W$ can be absorbed into $W'_{B,A}$. 

Similarly we have $W' \cup W'_{A,B}$ is homeomorphic to $W'_{A,B}$. That is, 
\[
W'_{A,B} \cong W'_{A,B} \cup W' \cong
W \cup W' \cup \Big( \bigcup_i V_i \Big) \cong 
W \cup W' \cup \Big( \bigcup_i V_i' \Big) \cong
W \cup W'_{B,A} \cong W'_{B,A}.
\]
This finishes the proof. 
\end{proof}

Going forward, we will use $W_{A,B}$ to denote the (well-defined up to homeomorphism) sets constructed in the Lemma, each containing a single representative of each of its maximal countable types.

%-------------------------------------------------------
\subsection{Classification of CB generated mapping class groups} \label{sec:tameCBgen} 

The purpose of this section is to prove Theorem \ref{thm:CB_generated}, namely, the statement that the necessary conditions for CB generation introduced in Section \ref{sec:criteria} are also sufficient for tame surfaces.

We continue with the notation and conventions introduced in the previous section, in particular the following.

%%%%%%%%%%%%%%%%%%
\begin{convention}
Going forward, $L$ denotes the finite-type surface furnished by Proposition \ref{prop:partition}, so that the complimentary regions to $L$ produce a decomposition 
$E = \bigsqcup_{A \in \calA} A$
where each $A$ is self-similar, and we have $\sqcup \calM(A) = \calM(E)$.  
\end{convention} 
The next proposition is the main technical ingredient in the proof of Theorem \ref{thm:CB_generated}.  It says that, by using elements from a CB set, one may map any neighborhood of $x_A$ in $E$ homeomorphically onto $A$ while pointwise fixing any set $B \in \calA$ which shares no end types with $A-U$.    
%%%%%%%%%%%%%%%%%%%%%%%%%%%%%%

\begin{proposition} \label{prop:push} 
Assume $E$ is tame, not of limit type 
and $\mcg(\Sigma)$ does not have infinite rank.  
Then there is a finite set 
$F \subset \mcg(\Sigma)$ such that the following holds: \\
Let $A \in \calA$ and let $U \subset A$ be a neighborhood of $x_A$.  
If $\calB_U \subset \calA$ is a subset that satisfies $E(y) \cap ( \bigcup_{B \in \calB_U} B )\neq \emptyset$ for all $y \in A - U$, then 
there is an element $f$ in the group generated by $F$ and $\calV_L$ with 
$f(U) = A$, and $f|_C = \id$ for all $C \in (\calA-\calB_U)$.  
\end{proposition}

\begin{proof}
The proof consists of several preliminary structural results on end spaces, carried out in Steps 1-4; the set $U$ and $\calB_U$ are introduced in the final step.  

\subsection*{Step 1: Decomposition of the sets $\mathbf{A \in \calA}$.} 
Fix $A \in \calA$.  For every $B \in \calA$, consider a copy of 
$W_{A,B} \subset A$ as in Lemma~\ref{lem:best_set}, as well as a homeomorphic copy of $W_A$.  A short argument shows that 
we may choose these sets to be pairwise disjoint, so that we have 
$W_{A,B} \cap W_{A, B'} = \emptyset$ whenever $B \neq B'$ and 
$W_{A,B} \cap W_A = \emptyset$ for all $B$. This is as follows: enumerate the 
sets $B_1, B_2, \ldots, B_k$ of $\calA - \{A\}$ and perform our original construction to obtain $W_{A,B_1}$.  
This set is disjoint from $N(x_A)$.  By self-similarity, there is a homeomorphic copy of $A$ inside $N(x_A)$, hence we may find a set $W_{A, B_2}$ disjoint from $W_{A, B_1}$ and also disjoint from a smaller copy of $N(x_A)$.  Continuing in this manner, we may produce the desired sets. Doing this one more time, we also find a disjoint copy of $W_A$.  
We keep these sets (and refer to them to by this notation, $W_{A,B}$ and $W_A$) 
for the remainder of the proof.  

Let 
\[
T_0 =W_A \sqcup \left(\bigsqcup_{B \in \calA-\{A\}} W_{A, B}\right) \subset A. 
\]
By construction, for every $y \in A- \{x_A\}$, $E(y)$ intersects $T_0$ by Theorem \ref{thm:loc_CB}.

Let $V_1 = A - T_0$ and consider a family of neighborhoods $V_k \searrow x_A$. Each $V_k$ contains a copy of $A$ and hence a copy $T_k$ of $T_0$. After dropping some of the sets $V_k$ 
from the nested sequence and reindexing, we can assume $T_1 \subset (V_1 - V_2)$. 
Continuing in this way, we find a new nested sequence of neighborhoods (which we again denote 
by $V_k$) so that 
$(V_k - V_{k+1})$ contains a copy $T_k$ of $T_0$. In particular, the sets $T_k$ are disjoint. 

Our next goal is to modify this construction so that we in fact have $(V_k - V_{k+1}) \cong T_k$, i.e. we obtain a nested family of neighborhoods such that the annular regions between them are homeomorphic to the sets $T_k$ above.  
To do this, we first show that we can distribute the set 
\[
Q = (V_1-V_2)-T_1
\] 
among finitely many of the other sets $T_k$, $k>1$ while preserving the homeomorphism class of the $T_k$; and then proceed iteratively. 

For each point $y \in Q$, $E(y)$ intersects $T_0$ and hence $y$ has a neighborhood 
$V_y \subset Q$ that has a homeomorphic copy inside $T_0$. Since $Q$ is compact, 
finitely may such neighborhoods are sufficient 
to cover $Q$. Making some of these neighborhoods smaller, we can 
write $Q = Q_1 \sqcup \dots \sqcup Q_m$, where every $Q_i$ has a copy 
in $T_0$ and hence in every $T_k$. For $j=1, \dots, m$ and $k \equiv j \mod m$ 
let $Q_k'$ be the copy of $Q_j$ in $V_k$. For $k=1, \dots, m$ define
\[
T_k' = (T_k -Q_k') \cup Q_k
\]
and for $k >m$, define 
\[
T_k' = (T_k - Q_k') \cup Q_{k-m}'. 
\] 
Each $T_k'$ is still homeomorphic to $T_0$, the sets $T_k'$ are disjoint and every point in 
$(V_1 - V_2)$ is contained in some $T_k'$. Note that $T_0$ is not modified. 

Similar to above, we can distribute the points in 
\[
Q'= (V_2-V_3) - \bigcup_{k \geq 1} T_k'
\] 
among the sets $T_k'$, $k=2, 3, \dots$ without changing their topology. 
That is, we obtain a family $T_k''$ of disjoint
sets homeomorphic to $T_0$ such that their union covers $A - V_3$ without
modifying $T_0$ or $T_1'$. Continuing in this way, every $T_k$ is modified 
finitely many times and stabilizes after $k$ steps. Thus, $\{ T_k^{(k)} \mid k \in \N\}$ is a family 
of disjoint copies of $T_0$ that covers $A - \{x_A\}$. To simplify notation, denote $T_k^{(k)}$ by $T_k(A)$. 
To summarize, 
\[
A-\{x_A\} = \bigsqcup_{k\geq 0} T_k(A), 
\qquad\text{and for}\qquad
U_n = \bigsqcup_{k\geq n} T_k(A), \text{ we have } U_n \searrow \{x_A\}. 
\]

Since $T_0 = W_A \sqcup \left(\sqcup_{B\neq A} W_{A,B}\right)$, we have a similar 
decomposition each set $T_k(A)$, which we notate by  
\[
T_k(A) = W_A^k \sqcup \left(\bigsqcup_{B \in \calA - \{ A\}} W_{A, B}^k\right). 
\]
where $W_A^k$ is a set homeomorphic to $W_A$ and 
$W_{A, B}^k$ is a set homeomorphic to $W_{A,B}$, for $k \in \N$.  

We also have the above decomposition for every $B \in \calA-\{A\}$.  For notational convenience, when $k<0$,
we define
\[
W_{A,B}^k := W_{B, A}^{-k-1}. 
\]

\subsection*{Step 2: a first shift map.} 
Using the decomposition above, we define the first (of several) homeomorphisms that shifts points between $A$ and $B$.  
Since the sets $W_{A,B}^k$, for $k \in \Z$ are disjoint and homeomorphic and Hausdorff converge to the points $x_A$ and $x_B$ as $k$ approaches $\infty$ and $-\infty$ respectively, there exists a homeomorphism  
 $\eta_{A,B}$ such that 
\[
\eta_{A,B} \big(W_{A, B}^k\big) = W_{A, B}^{k+1}, \qquad \forall \, k \in \Z
\]
and restricts to the identity elsewhere in $E$.   Fix one such map for each 
(unordered) pair $A, B \in \calA$. 

\subsection*{Step 3: shifting maximal countable ends independently.} 
Now we define homeomorphisms allowing one to shift the maximal countable ends one by one.  As motivation, 
consider, for instance, a surface with $E \cong \omega \cdot 2 +1$, such that $E^G$ and the closure of $E-E^G$ are both homeomorphic to $\omega \cdot 2+ 1$.   There are two maximal ends, $\calA = \{A, B\}$ and we have the simple situation where $W_{A,B} = T_0$ consists of one of each type of non-maximal end.  The map $\eta_{A,B}$ moves both to the right or to the left.  
However, there is evidently a homeomorphism of $\Sigma$ which pointwise fixes $E- E^G$ and shifts the nonmaximal ends of $E^G$.   

  \begin{figure*}[h]
     \centerline{ \mbox{
 \includegraphics[width = 4in]{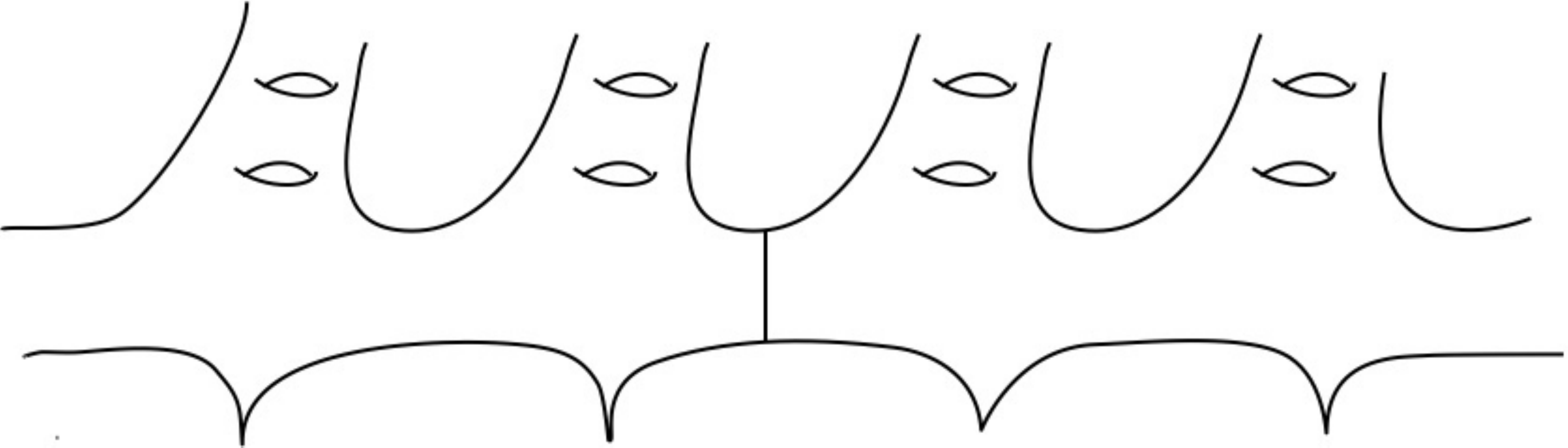}}}
 \caption{$E - E^G$ and the non-maximal ends of $E^G$ can be shifted independently}
  \label{fig:shift_ends}
  \end{figure*}

For $z \in E_{\rm mc}(A,B)$, 
let $W_{A, B}^k(z) \subset W_{A,B}^k$ be a stable neighborhood of the unique 
intersection point of $E(z)$ with $W_{A, B}^k$. By making these
neighborhood smaller, we can assume $W_{A, B}^k(z)$ for different $z \in E_{\rm mc}(A,B)$ 
are disjoint. 
(This is a slight abuse of notation since $W_{A, B}^k(z)$ depends only 
on the equivalence class of $z$ under $\sim$, not the point itself).  
Define $\eta_{A,B,z}$ to be a homeomorphism of $\Sigma$ so that 
\[
\eta_{A,B,z} \big(W_{A, B}^k(z) \big) = W_{A, B}^{k+1}(z), \qquad k \in \Z
\]
and that acts by the identity elsewhere in $E$. Note that the actions of $\eta_{A,B,z}$ on $E$ commute with each other 
and have support in $A \cup B$.

\subsection*{Step 4: Standard decomposition for sets of shared ends.} 
The following claim shows that clopen subsets of $E(A,B)$ have a standard form.  

\begin{claim}
Let $W \subset E(A,B)$ be a clopen set containing $W_{A,B}$ and disjoint from $x_A$. 
For $z \in E_{\rm mc}(A,B)$, let $p_z(W) = |E(z) \cap W|$. Then $W$ is homeomorphic to the set
\[ W_{A, B} \sqcup 
  \left( \bigsqcup_{[z] \in E_{\rm mc}(A,B)} \bigsqcup_{k=1}^{p_z(W)-1}  W_{A, B}^k(z) \right).
\]
\end{claim}

Recall that $W_{A,B} \subset T_0$ was a fixed set, chosen in step 1. However, note that this structure 
theorem also applies to any clopen subset of $E(A,B)$ which contains a {\em homeomorphic 
copy} of $W_{A,B}$.  

\begin{proof}[Proof of claim]  \renewcommand{\qedsymbol}{$\blacksquare$}
For $z \in E_{\rm mc}(A,B)$ and $y \in E(z) \cap (W- W_{A,B})$, choose a stable 
neighborhood $V_y$ of $y$ in $W$. Making the neighborhoods small enough, we can 
assume they are disjoint from each other and from $W_{A,B}$. Since stable neighborhoods are canonical, we can map the union of these 
neighborhoods homeomorphically to 
\[
\bigsqcup_{[z] \in E_{\rm mc}(A,B)} \bigsqcup_{k=1}^{p_z(W)-1}  W_{A, B}^k(z).
\]
It remains to show that if $p_z(W) = 1$ for every $z \in E_{\rm mc}(A,B)$, then $W$ is 
homeomorphic to $W_{A,B}$. 

For every point in $y \in (W-W_{A,B})$, there is a point $x \in W_{A,B}$ where 
$x$ is an accumulation point of $E(y)$. By Lemma~\ref{lem:consume}, for any stable 
neighborhoods $V_x$ and $V_y$ of $x$ and $y$ respectively, $V_x \cup V_y$ is 
homeomorphic to $V_x$.  Taking a cover of $W-W_{A,B}$ by such neighborhoods, we 
conclude that 
\[
W = (W- W_{A,B}) \cup W_{A,B} \cong W_{A,B}. 
\]
This proves the claim. 
\end{proof} 

\subsection*{Step 5. Finishing the proof.} Let 
\[
F = \Big\{ \eta_{A,B}^{\pm1}, \, \eta_{A,B,z}^{\pm 1} \, \Big| \, 
B \in \calA-\{A\} \quad\text{and}\quad [z] \in E_{\rm mc}(A,B)\Big\}.
\]
Let $U \subset A$ be a neighborhood of $x_A$ and let $\calB_U \subset \calA - \{A\}$ be 
as in the statement of the proposition. 
The homeomorphism $\prod_{B \in \calB_U} \eta_{A, B}^{-1}$ shifts the sets $W_{B,A}$
from $\sqcup_{B \in \calB_U}B$ into $A$, and in particular, 
\[
\bigsqcup_{B \in \calB_U} W_{A,B} \subset 
 A- \left( \prod_{B \in \calB_U} \eta_{A, B}^{-1}\right) (U).
\] 
Thus, up to applying this homomorphism, we may assume that $U$ is sufficiently small so that its complement contains $\sqcup_{B \in \calB_U} W_{A,B}$, the subset of $T_0$.

Fix $B_1 \in \calB_U$.  
Since $(A-U)\cap E(A,B_1)$ contains $W_{A, B_1}$, the claim proved in step 4 implies that $(A-U)\cap E(A,B_1)$
it is homeomorphic to the standard set 
\[ W_{A, B_1} \sqcup 
  \left( \bigsqcup_{[z] \in E_{\rm mc}(A,B_1)} \bigsqcup_{k=1}^{p_z(W)-1}  W_{A, B_1}^k(z) \right)
\]
in $A$ and the complements of these two sets in $A$ are homeomorphic
(each being homeomorphic to $A$). Thus, by the classificaiton of surfaces there is a homeomorphism $v_1$ supported on the complementary region to $L$ with end space $A$, hence in $\calV_L$, 
taking $(A- U) \cap E(A,B)$ to this standard set.  However, by construction, the image 
of this standard set under 
\[
\eta_{A,B_1} \circ \prod_{z \in \sE_c(A,B_1)} \eta_{A,B_1, z}^{p_z(W)-1}
\] 
is disjoint from $A$, and the image of its complement in $A$ is equal to $A$. 
Let 
\[
U' = \eta_{A,B_1} \circ \prod_{z \in \sE_c(A,B_1)} 
 \eta_{A,B_1, z}^{p_z(W)-1} \circ v_1(U).
 \]
Note that $\calB_{U'} = \calB_{U}- \{ B_1\}$. We now repeat the process above using 
$B_2 \in \calB_{U'}$ and $U'$ and produce an element of the subgroup generated by $F$ 
and $\calV_L$ which takes $U'$ to a subset of $A$ containing $E(A, B_2)$.  
Iterating this process for each $B \in \calB_U$ achieves the desired result.  
\end{proof}

We are almost ready to prove the main result of this section.  In order to do so, we need another finite set of mapping classes, the {\em handle shifts}, which we define now.   
See also \cite[Section 6]{PatelVlamis} for earlier use of this class of maps.  

\begin{definition} 
An {\em infinite strip with genus} is the surface $\R \times [-1,1]$ with a handle attached to the interior of each set $[m,m+1] \times [0,1]$, such that $(x, y) \mapsto (x+1, y)$ is a homeomorphism of the surface.  

A {\em handle shift} on the infinite strip with genus is the mapping class of the homeomorphism $h$ which pointwise fixes the boundary, agrees with $(x,y) \mapsto (x+1, y)$ outside an $\epsilon$-neighborhood of the boundary, and on the $\epsilon$ neighborhood agrees with $(x,y) \mapsto (x+\frac{1-|y|}{\epsilon}, y)$.  
\end{definition}

\begin{definition} \label{def:handle_shift}
Suppose that $\Sigma$ has locally CB mapping class group and $L$ is a surface as in Lemma \ref{lem:CB_gen_nbhd}   We call a (infinite type) subsurface $R \subset \Sigma$ an {\em infinite strip with genus} if it is homeomorphic to such a surface {\em and} so that each complimentary region $\Sigma'$ to $L$ which contains an end of $R$ has infinite genus in $\Sigma' - R$. 

A {\em handle shift} on $R$ is the mapping class of the map $h$ above (under our identification), extended to agree with the identity on the complement of $R$.  
\end{definition}

 Recall that the {\em pure mapping class group}, denoted $\pmcg(\Sigma)$, is the subgroup of $\mcg(\Sigma)$ which pointwise fixes $E$.  
We now prove a lemma on generating pure mapping classes. 

For each pair $(A,B)$ such that $x_A$ and $x_B$ are both accumulated by genus, let $R_{AB} \subset \Sigma$ be an infinite strip with genus with one end in $A$ and one end in $B$.  We may choose these (one at a time) so that they are disjoint subsurfaces of $\Sigma$.  
Fix also a handle shift $h_{AB} \in \Homeo(\Sigma)$ on $R_{AB}$.  

\begin{lemma}[Generating $\pmcg(\Sigma)$] \label{lem:gen_pmap}  
Let $G$ be a subgroup of $\mcg(\Sigma)$ containing all mapping classes supported on finite type subsurfaces, all mapping classes that fix each of the boundary components of $L$ and the handle shifts $h_{AB}$ defined above.  Then $G$ contains $\pmcg(\Sigma)$.  
\end{lemma} 

\begin{proof}
For $A \in \calA$, let $\Sigma_A$ denote the connected component of $\Sigma - L$ with end space $A$, and let $\partial_A$ denote its boundary component.  
Let $g \in \pmcg(\Sigma)$.  Then $g(\Sigma_A)$ also has end space $A$, and a single boundary component $g(\partial A)$.    Let $T \subset \Sigma$ be a connected, finite type subsurface large enough to contain $L \cup g(L)$.  
If, for each $A \in \calA$, the surface $\Sigma_A \cap T$ is homeomorphic rel $\partial T$ to $g(\Sigma_A) \cap T$, then there is a mapping class $\phi$ supported on $T$ such that $\phi(L) = L$, which proves what we needed to show.  

So we are reduced to the case where for some $A$ the surface $\Sigma_A \cap T$ is not homeomorphic to $g(\Sigma_A) \cap T$.   Both are connected surfaces with the same number of boundary components, so we conclude that they must have different genus.
In particular, this only occurs if $\Sigma$ is itself of infinite genus, for otherwise we choose $K$ by convention to contain all the genus of $\Sigma$.  

Without loss of generality, assume that the genus of $g(\Sigma_A) \cap T$ is larger than that of $\Sigma_A \cap T$.  Since $T$ is finite genus, there must also be another $B \in \calA$ such that the genus of $g(\Sigma_B) \cap T$ is smaller than that of $\Sigma_B \cap T$. 
Since $L$ is chosen so that complimentary regions have either zero or infinite genus, we conclude that $\calM(A)$ and $\calM(B)$ must be accumulated by genus. 

Consider the handle shift $h_{AB}$ supported on $R_{AB}$, which has one end in $A$ and one end in $B$.  
Let $\phi$ be a homeomorphism preserving the ends of $\Sigma$, pointwise fixing the boundary components of $L$, and such that the intersection of $\phi(R_{AB})$ with $T \cap (g(\Sigma_A) - \Sigma_A)$ and with $T \cap (\Sigma_B \cap g(\Sigma_B))$ each have genus one,  that $\phi(R_{AB}) \cap T$ has genus 2 (so there is not genus elsewhere in $T$), and so that, up to replacing $h_{AB}$ with its inverse, $\phi h_{AB} \phi^{-1}$ shifts the genus from $T \cap g(\Sigma_A)$ into $T \cap \Sigma_B$.  See figure \ref{fig:handleshift} for an illustration in a simple setting. 
Such a homeomorphism $\phi$ exists by the classification of surfaces, and our stipulation that the compliment of $R_{AB}$ have infinite genus in complimenatry regions of $L$.  

  \begin{figure*}[h]
     \labellist 
  \small\hair 2pt
   \pinlabel $\Sigma_A$ at -15 230
   \pinlabel $\Sigma_B$ at 602 220
   \pinlabel $g(K)$ at 420 160
   \pinlabel $K$ at 190 50
   \endlabellist
     \centerline{ \mbox{
 \includegraphics[width = 3.5in]{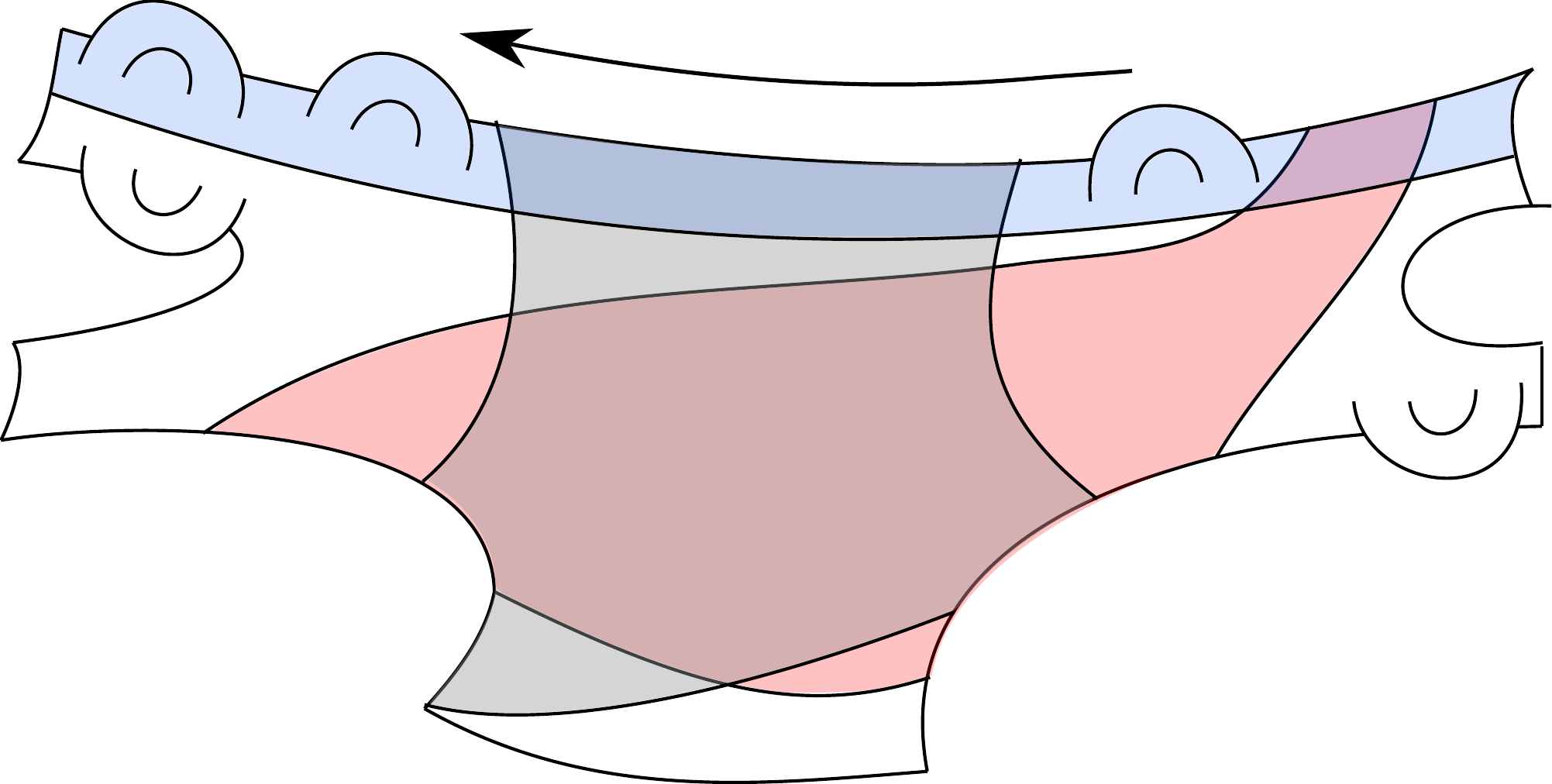}}}
 \caption{$T$ containing $K$ and $g(K)$, and the domain $\phi(R_{AB})$ of the handle shift}
  \label{fig:handleshift}
  \end{figure*}
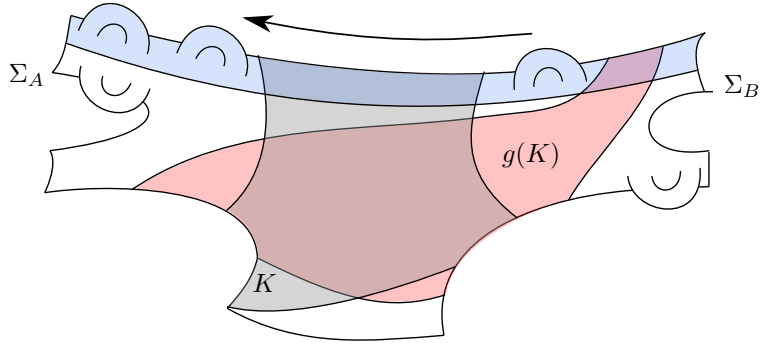

Then $\phi h_{AB} \phi^{-1} g(\Sigma_A) \cap T$ has one fewer genus than $g(\Sigma_A) \cap T$, and 
$\phi h^{-1}_{AB} \phi^{-1} g(\Sigma_A) \cap T$ one additional genus, and there is no change otherwise in the genus of complimentary regions.  
Continuing in this fashion, one may iteratively modify $g$ by composing by elements of $G$ so as to arrive at a homeomorphism $g'$ with the property that 
$\Sigma_A \cap T$ is homeomorphic to $g'(\Sigma_A) \cap T$ for all $A \in \calA$, which is what we needed to show. 
\end{proof} 

\paragraph{A CB generating set.} 
We our now in a position to prove the Theorem on CB generation.  Our CB generating set will consist of 
$\calV_K$, and the finite set consisting of the Dehn twists $D$ from Observation~\ref{obs:Emc}, the finite set $F$ from Proposition \ref{prop:push}, the handle shifts $h_{AB}$, and a finite collection of homeomorphisms $g_{AB}$, one for each pair $A, B \in \calA$ such that $x_A$ and $x_B$ are of the same type. 

\begin{proof}[Proof of Theorem \ref{thm:CB_generated}]
One direction follows from Lemmas \ref{lem:limit-type}  and \ref{lem:infinite-rank}. 
We prove the other direction.  For this, we show the generating set described in the paragraph above (with precise definitions of $g_{AB}$) is in fact CB.  

Let $\calV_K \cup D$ be the CB set given by Observation~\ref{obs:Emc} (recall that $D$ is a finite collection of Dehn twists).
Let $F$ be the finite set from Proposition \ref{prop:push}.   For each pair of maximal points $x_A, x_B$ in $E^G$, let $h_{AB}$ be the handle shift defined above Lemma \ref{lem:gen_pmap}.  
Let $\chi$ be the CB set consisting of $\calV_K \cup D$ together with the homeomorphisms from $F$ and all the $h_{AB}$.  

We show first that $\chi$ generates the pointwise stabilizer of $\{x_A : A \in \calA\}$.  After this, we will add finitely many more homeomorphisms $g_{AB}$ to generate $\mcg(\Sigma)$.  

Suppose that $\phi$ fixes each basepoint $x_A \in E$.  We proceed inductively on the number of elements of $\calA$ which are pointwise fixed by the action of $\phi$ on $E$.  Let $\calA_{\rm id}$ denote the subset (possibly empty) of $\calA$ such that, for each $A \in \calA_{\rm id}$, the ends of $A$ are pointwise fixed by $\phi$, and let $\calA^c = \calA - \calA_{\rm id}$.  
Choose a set $A \in \calA-\calA_{\rm id}$.   For every 
$B \in  \calA_{\rm id}^c $, let  $U_B=B - \phi(A)$. Then for every end $z \in (B-U_B) \subset \phi(A)$, there is some 
end $y \sim z$ which lies in $A$. Hence, by Proposition \ref{prop:push}
(setting $\calB = \calA-\{A, B\}$), there is an element $g$ in the group generated by $\chi$ with support
in $A \cup B$ that sends $U_B$ to $B$. In particular, $g\, \phi(A) \cap B = \emptyset$
and the restriction of $g \, \phi$ to sets in $\calA_{\rm id}$ is still the identity. 

Doing this for every other $C \in \calA_{\rm id}^c$, we may modify $\phi$ by elements of $\chi$ to obtain
 a map $\phi'$ such that  
$\phi'(A)$ is disjoint from every $C \in \calA-\{A\}$, i.e. $\phi'(A) \subset A$. 
Letting $U = \phi(A')$, we see that the conditions of the Proposition \ref{prop:push} 
are again satisfied since the restriction of $\phi'$ to sets in $\calA_{\rm id}$ is the identity. 
Hence, there is $g' \in \langle \chi \rangle$ that is also identity on every set in $\calA_{\rm id}$
and that sends $U$ to $A$. That is, $g' \phi(A) = A$ and there is $\psi \in \calV_L$
such that the restriction of $\psi g' \phi'$ to $A$ is the identity. 

Continuing in this way, at every step, we increase the number of sets in $\calA_{\rm id}$, eventually obtaining a homeomorphism which pointwise fixes all ends.  Since $\chi$ generates $\pmcg(\Sigma)$, we conclude that $\phi \in \langle \chi \rangle$.  

Now we show that there is a finite set $F'$ such that $\chi \cup F'$ generate $\mcg(\Sigma)$.  
Construct $F'$ as follows.  For any $A, B \in \calA$ where points in 
$\calM(A)$ and $\calM(B)$ are of the same type,  choose one element $g_{A, B}$
sending $N(x_A)$ to $N(x_B)$ (recall that these are stable neighborhoods) and restricting to the identity on every set in $\calA - \{ A, B\}$. 
Let $F'$ be the set of all such chosen $g_{A,B}$.   To see $\chi \cup F'$ generates, let $\phi \in \mcg(\Sigma)$.  Suppose 
$\phi(x_A) \in B$.  We modify $\phi$ to a map $\phi'$ in one of the following ways. 

(Case 1) Assume $\phi(x_B) \not = x_B$. There is a $\psi \in \calV_L$ with support
 in $B$ that sends $\phi(x_A)$ to $x_B$ and hence 
 \[
\phi' =  g_{A, B}  \psi  \phi
 \]
 fixes $x_A$. 

(Case 2) Assume $\phi(x_B) = x_B$. Then $\calM(B)$ has more than one point and hence
it is a Cantor set. Take a map $\psi \in \calV_L$ with support in $B$ that sends 
$\phi(x_A)$ to $x_B$ and sends $x_B$ to a point in $B-N(x_B)$. 
Then 
\[
\phi' =\psi^{-1}  g_{A,B}  \psi  \phi 
\] 
sends $x_A$ to $x_A$ and still fixes $x_B$. 

Note that the number of points $x_A$ that are fixed 
by $\phi'$ is one more than that of $\phi$. Hence, after repeating this process finitely many times, we arrive 
at an element fixing each maximal point, hence generated by $\chi$. This finishes the proof. 
\end{proof}

\section{Classification of CB Mapping class groups}  \label{sec:global_CB}
In this section we prove Theorem \ref{thm:global_CB} classifying the surfaces $\Sigma$ 
for which the group $\mcg(\Sigma)$ is CB. In the case where $E$ is uncountable, we will 
add the hypothesis that $\Sigma$ is tame. However, we expect the classification theorem 
to hold without this additional hypothesis, since it is only used in the very last portion of the proof.  

\noindent Note that the telescoping case occurs only when $E$ is uncountable, by Proposition \ref{prop:telescoping_uncountable}. 

\begin{proof}[Proof of Theorem \ref{thm:global_CB}]
If $\Sigma$ has zero or infinite genus and is either telescoping or has self-similar end space, then 
it was shown in Propositions \ref{prop:SSCB} and \ref{prop:telescoping} that 
$\mcg(\Sigma)$ is CB, with no hypothesis on tameness. 
We prove the other direction.  Assume that $\Sigma$ has a CB 
mapping class group.  By Example \ref{ex:finite-genus}, this implies that $\Sigma$ has zero or infinite genus.  Also, being globally CB, $\mcg(\Sigma)$ is in particular locally CB so the end space admits a decomposition $E = \sqcup_{A \in \calA} A$ into finitely many self-similar sets as in Theorem \ref{thm:loc_CB}. Then Example \ref{ex:easy_nondisplace} implies that, if we take such a decomposition with $\calA$ of minimal cardinality, then $\calA$ has either one or two elements.  Finally, if $\calA$ is a singleton, then $E$ is self-similar.  Thus, we only need to take care of the case where $\calA$ has exactly two elements.  

Example \ref{ex:easy_nondisplace} also shows that, if $\calA = \{A, B\}$, then $\calM(A)$ and $\calM(B)$ are either both singletons or Cantor sets.  A slight variation on the argument there also allows us to eliminate the case where they are both Cantor sets: if points of $\calM(A)$ are not of the same type as those in $\calM(B)$, then one may construct a nondisplaceable subsurface just as in the example by having $\calM(A)$ play the role of the singleton.  Otherwise, points of $\calM(A)$ and $\calM(B)$ are all of the same type and hence $\calM(E) = \calM(A) \cup \calM(B) = E(x_A)$ and Lemmas \ref{lem:SS_nbhd} and \ref{lem:consume} together imply that $E$ is self-similar.

Thus, we can assume that $\calM(A) = \{x_A\}$ and $\calM(B) = \{x_B\}$.  
We start by showing in this case that $E_\mathrm{mc}(A,B) = \emptyset$. 
To show this, suppose for contradiction that we have some $z \in E_\mathrm{mc}(A,B)$.  Then $E(z)$ accumulates to both $x_A$ 
and $x_B$ and since $z$ is maximal in $E - \{x_A, x_B\}$, the set $E(z)$ has no other accumulation
points. As in Lemma \ref{lem:infinite-rank}, we can 
define a continuous homomorphism to $\Z$ on the subgroup that pointwise fixes $\{x_A, x_B\}$ (which is of index at most two in $\mcg(\Sigma)$), via
%\begin{equation} \label{eq:ell}
\[
\ell(\phi) = 
\Big| \big\{x \in E(z) \mid x \in A, \quad \phi(x) \in B \big\}\Big| -
\Big| \big\{x \in E(z) \mid x \in B, \quad \phi(x) \in A \big\}\Big|. 
\]
%\end{equation}
Let $U_0 \subset A$ be a neighborhood of $z$ not containing $x_A$. 
Since $z \in E_{\rm mc}(A,B)$, we can find a homeomorphic copy
$U_1 \subset B$ of $U_0$ in $B$. Since $A$ and $B$ are self-similar, we may find 
disjoint homeomorphic copies $U_2, U_3, \ldots$ of $U_0$ in $A$ descending to $x_A$, 
and homeomorphic copies $U_{-1}, U_{-2}, \ldots$ of $U_0$ in $B$ descending to $x_B$.  
Let $\eta$ be a homeomorphism that sends $U_i$ to $U_{i+1}$ 
and restricts to the identity everywhere else. Then $\ell (\eta^n)=n$, so the homomorphism 
$\ell$ is unbounded and $\mcg(\Sigma)$ is not CB. This gives the desired contradiction, so we conclude that  
$E_\mathrm{mc}(A,B) = \emptyset$. Note that, in particular, this implies 
$E$ is not countable. 

We now show that $E$ is telescoping.  Let $N(x_A)$ and $N(x_B)$ be as in Lemma  \ref{lem:W_sets}.  Let $V_1$ and $V_2$ be subsurfaces with a single boundary component so that the end space of $V_1$ is $N(x_A)$ and that of $V_2$ is $N(x_B)$.  
We will check the definition of telescoping by using these neighborhoods of $x_1 = x_A$ and $x_2 = x_B$. 

Let $W_1 \subset V_1$ and $W_2 \subset V_2$ be neighborhoods of $x_A$ and 
$x_B$ respectively. Let $S$ be a finite type subsurface, homeomorphic to a pair of pants, 
whose complimentary regions partition $E$ into $W_1$, $V_2$ and the remaining ends.    Provided $N(x_A)$ and $N(x_B)$ are chosen small enough,  condition (3) of Theorem \ref{thm:loc_CB} ensures that either $\Sigma$ has genus 0, or that $\Sigma - \left(V_1 \cup V_2\right)$ has infinite genus. 

Let $f_1$ be a homeomorphism displacing $S$.  We may also assume that $f_1$ fixes $x_A$ and $x_B$, since existence of a nondisplaceable subsurface in the finite-index subgroup of $\mcg(\Sigma)$ stabilizing $x_A$ and $x_B$ is sufficient to show that $\mcg(\Sigma)$ is not CB.  
Then, up to replacing $f_1$ with its inverse, we have $f_1(\Sigma - W_1) \subset V_2$.  A similar argument gives a homeomorphism $f_2$ with $f_2(\Sigma - W_2) \subset V_1$ and so the second condition in the definition of telescoping is satisfied.  

For the first condition, we need to find a homeomorphism of the subsurface $\Sigma - V_2$ that maps $W_1$ to $V_1$.  By Lemma \ref{lem:consume}, we know that $V_1$ and $W_1$ are homeomorphic (their end sets are homeomorphic, and they each have zero or infinite genus and one boundary component)
 so we need only show that their compliments are homeomorphic and apply the classification of surfaces.   Since, as remarked above, $\Sigma$ either has genus 0 or $\Sigma - \left(V_1 \cup V_2\right)$ has infinite genus, we need only produce such a homeomorphism on the level of end spaces.   Here we will finally invoke tameness. Let 
\[\Sigma' = \Sigma- \big(V_1 \cup V_2 \big).\]  
By definition of $N(x_A)$, for any end $z$ of $V_1 - W_1$ there exists an end $x$ of $\Sigma'$ where $z \lesseq x$, and hence some maximal point $x \in V$ with $z \lesseq x$.  
Since $x$ is not of countable type, it is necessarily an accumulation point of $E(x)$.  Since $x$ has a stable neighborhood $V_x$, Lemma \ref{lem:consume} implies that $z$ has a neighborhood $U_z$ such that $U_z \cup V_x$ is homeomorphic to $V_x$.  Thus, on the level of ends, the end space of $\Sigma'$ is homeomorphic to that of its union with $U_z$.  

Since the end space of $(V_1 - W_1)$ is compact, it may be covered by finitely many such neighborhoods $U_z$ (varying $z$); applying the procedure above to each of them in turn produces the desired homeomorphism on the level of end spaces, showing the two subsurfaces are homeomorphic.  
\end{proof}

\bibliographystyle{plain}

\end{document}